\documentclass{article}%
\usepackage{graphicx}
\usepackage{amsmath}
\usepackage{amsfonts}
\usepackage{amssymb}
\usepackage{setspace}
\usepackage[hmargin=3.5cm,vmargin=3.5cm]{geometry}%
\providecommand{\U}[1]{\protect\rule{.1in}{.1in}}
\providecommand{\U}[1]{\protect\rule{.1in}{.1in}}
\providecommand{\U}[1]{\protect\rule{.1in}{.1in}}
\providecommand{\U}[1]{\protect\rule{.1in}{.1in}}
\providecommand{\U}[1]{\protect\rule{.1in}{.1in}}
\providecommand{\U}[1]{\protect\rule{.1in}{.1in}}
\providecommand{\U}[1]{\protect\rule{.1in}{.1in}}
\providecommand{\U}[1]{\protect\rule{.1in}{.1in}}
\providecommand{\U}[1]{\protect\rule{.1in}{.1in}}
\providecommand{\U}[1]{\protect\rule{.1in}{.1in}}
\providecommand{\U}[1]{\protect\rule{.1in}{.1in}}
\providecommand{\U}[1]{\protect\rule{.1in}{.1in}}
\providecommand{\U}[1]{\protect\rule{.1in}{.1in}}
\providecommand{\U}[1]{\protect\rule{.1in}{.1in}}
\providecommand{\U}[1]{\protect\rule{.1in}{.1in}}
\providecommand{\U}[1]{\protect\rule{.1in}{.1in}}
\providecommand{\U}[1]{\protect\rule{.1in}{.1in}}
\providecommand{\U}[1]{\protect\rule{.1in}{.1in}}
\providecommand{\U}[1]{\protect\rule{.1in}{.1in}}
\providecommand{\U}[1]{\protect\rule{.1in}{.1in}}
\providecommand{\U}[1]{\protect\rule{.1in}{.1in}}
\providecommand{\U}[1]{\protect\rule{.1in}{.1in}}
\providecommand{\U}[1]{\protect\rule{.1in}{.1in}}
\providecommand{\U}[1]{\protect\rule{.1in}{.1in}}
\providecommand{\U}[1]{\protect\rule{.1in}{.1in}}
\providecommand{\U}[1]{\protect\rule{.1in}{.1in}}
\providecommand{\U}[1]{\protect\rule{.1in}{.1in}}
\providecommand{\U}[1]{\protect\rule{.1in}{.1in}}
\providecommand{\U}[1]{\protect\rule{.1in}{.1in}}
\providecommand{\U}[1]{\protect\rule{.1in}{.1in}}
\providecommand{\U}[1]{\protect\rule{.1in}{.1in}}
\providecommand{\U}[1]{\protect\rule{.1in}{.1in}}
\providecommand{\U}[1]{\protect\rule{.1in}{.1in}}
\providecommand{\U}[1]{\protect\rule{.1in}{.1in}}
\providecommand{\U}[1]{\protect\rule{.1in}{.1in}}
\providecommand{\U}[1]{\protect\rule{.1in}{.1in}}
\providecommand{\U}[1]{\protect\rule{.1in}{.1in}}
\providecommand{\U}[1]{\protect\rule{.1in}{.1in}}

\newtheorem{theorem}{Theorem}
{}

\newtheorem{corollary}{Corollary}

\newtheorem{definition}{Definition}

\newtheorem{notation}{Notation}

\newtheorem{proposition}{Proposition}
\newtheorem{remark}{Remark}

\newtheorem{summary}{Summary}
\newenvironment{proof}[1][Proof]{\textbf{#1.} }{\ \rule{0.5em}{0.5em}}

\begin{document}

\title{The Spectrum of the Hamiltonian with a PT-symmetric Periodic Optical Potential}
\author{O. A. Veliev\\{\small Depart. of Math., Dogus University, Ac\i badem, Kadik\"{o}y, \ }\\{\small Istanbul, Turkey.}\ {\small e-mail: oveliev@dogus.edu.tr}}
\date{}
\maketitle

\begin{abstract}
We give a complete description, provided with a mathematical proof, of the
shape of the spectrum of the Hill operator with potential $4\cos^{2}%
x+4iV\sin2x$, where $V\in(0,\infty).$ We prove that the second critical point
$V_{2}$, after which the real parts of the first and second band disappear, is
a number between $0.8884370025$ and $0.8884370117.$ Moreover we prove that
$V_{2}$ is the degeneration point for the first periodic eigenvalue. Besides,
we give a scheme by which one can find arbitrary precise value of the second
critical point as well as the $k$-th critical points after which the real
parts of the $(2k-3)$-th and $\left(  2k-2\right)  $-th bands disappear, where
$k=3,4,...$



\end{abstract}

\section{Introduction and Preliminary Facts}

In this paper we investigate the one dimensional Schr\"{o}dinger operator
$L(q)$ generated in $L_{2}(-\infty,\infty)$ by the differential expression
\begin{equation}
-y^{^{\prime\prime}}(x)+q(x)y(x),
\end{equation}
where $q(x)=:4\cos^{2}x+4iV\sin2x$ is an optical potential which is the shift
of%
\begin{equation}
(1+2V)e^{i2x}+(1-2V)e^{-i2x},\text{ }V\geq0.
\end{equation}

Some physically interesting results have been obtained by considering the
potential (2). The detailed investigation of the periodic optical potentials
in the papers [7, 8] were illustrated on (2).\ For the first time, the
mathematical explanation of the nonreality of the spectrum of $L(q)$ for
$V>0.5$ and finding the threshold $0.5$ (first critical point $V_{1}$) was
done by Makris et al [7,8]. Moreover, for $V=0,85$ they sketch the real and
imaginary parts of the first two bands by using the numerical methods. In [11]
Midya et al reduce the operator $L(q)$ with potential (2) to the Mathieu
operator and using the tabular values establish that there is second critical
point $V_{2}\sim0.888437$ after which no part of the first and second bands
remains real.

In this paper we give a complete description, provided with a mathematical
proof, of the shape of the spectrum of the Hill operator $L(q)$ with potential
(2), when $V$ changes from $1/2$ to $\sqrt{5}/2.$ We prove that the second
critical point $V_{2}$ is a number between $0.8884370025$ and $0.8884370117$.
Moreover we prove that $V_{2}$ is the unique degeneration point for the first
periodic eigenvalue, in the sense that the first periodic eigenvalue of the
potential (2) is simple for all $V\in(1/2,\sqrt{5}/2)\backslash\left\{
V_{2}\right\}  $ and is double if $V=V_{2}.$ Our approach give the possibility
to find the arbitrary close values of the $k$-th critical point $V_{k}$ and
prove that no part of the $\left(  2k-3\right)  $-th and $\left(  2k-2\right)
$-th bands remains real for $V_{k}<V<V_{k}+\varepsilon$ for some positive
$\varepsilon,$ where $k=2,3,...$

For the proofs we use the following results formulated below as summaries.

\begin{summary}
$(a)$ The spectrum $\sigma(L(q))$ of $L(q)$ is the union of the spectra
$\sigma(L_{t}(q))$ of the operators $L_{t}(q)$ for $t\in(-\pi,\pi]$ generated
in $L_{2}[0,\pi]$ by (1) and the boundary conditions
\[
y(\pi)=e^{it}y(0),\text{ }y^{^{\prime}}(\pi)=e^{it}y^{^{\prime}}(0),
\]
where $t$ is the quasimomentum.

$(b)$ The spectrum of $\sigma(L_{t}(q))$ consists of the Bloch eigenvalues
$\mu_{1}(t),$ $\mu_{2}(t),...,$ that are the roots of the characteristic
equation
\begin{equation}
F(\lambda)=2\cos t,
\end{equation}
where $F(\lambda):=\varphi^{\prime}(\pi,\lambda)+\theta(\pi,\lambda)$ is the
Hill discriminant, $\theta$ and $\varphi$ are the solutions of
\begin{equation}
-y^{^{\prime\prime}}(x)+q(x)y(x)=\lambda y(x)
\end{equation}
satisfying the initial conditions $\theta(0,\lambda)=\varphi^{\prime
}(0,\lambda)=1,$ $\theta^{\prime}(0,\lambda)=\varphi(0,\lambda)=0.$

$(c)$ $\lambda\in\sigma(L(q))$ if and only if $F(\lambda)\in\lbrack$\ $-2,2]$.

$(d)$ The spectrum $\sigma(L)$ consists of the analytic arcs defined by (3)
whose endpoints are the eigenvalues of $L_{t}$ for $t=0,\pi$ and the multiple
eigenvalues of $L_{t}$ for $t\in(0,\pi).$ Moreover $\sigma(L)$ does not
contain the closed curves, that is, the resolvent set $\mathbb{C}%
\backslash\sigma(L)$ is connected.

$(e)$ If the potential $q$ is PT-symmetric, that is, $q(-x)=\overline{q(x)},$
then the following implications hold:
\begin{equation}
\lambda\in\sigma(L_{t}(q))\Longrightarrow\overline{\lambda}\in\sigma
(L_{t}(q))\text{ }\And\lambda\in\mathbb{R}\Longrightarrow F(\lambda
)\in\mathbb{R}.
\end{equation}

For $(a)-(c)$ see [3,9,10,13], for $(d)$ see [13], for $(e)$ see [7] and [14].
For the properties of the general PT-symmetric potentials see [1, 12 and
references of them]. Here we only note that the investigations of PT-symmetric
periodic potentials were begun by Bender et al [2].
\end{summary}

As we noted above in [11] it was proved that the investigation of the operator
$L(q)$ with potential (2) can be reduced to the investigation of the Mathieu
operator. \ Besides in [15] (see Theorem 1 and (26) of [15]) we proved that if
$ab=cd$, where $a,b,c,$ and $d$ are arbitrary complex numbers, then the
operators $L(q)$ and $L(p)$ with potentials $q(x)=ae^{-i2x}+be^{i2x}$ and
$p(x)=ce^{-i2x}+de^{i2x}$ have the same Hill discriminant $F(\lambda)$ and
hence the same Bloch eigenvalues and spectrum. Therefore we have
\begin{equation}
\sigma(L(V))=\sigma(H(a)),\text{ }\sigma(L_{t}(V))=\sigma(H_{t}(a)),\text{
}\forall t\in\lbrack0,\pi],\text{ }a=\sqrt{1-4V^{2}},
\end{equation}
where, for brevity of notations, the operators $L_{t}(q)$ ($L(q))$ with
potentials (2) and
\begin{equation}
q(x)=ae^{-i2x}+ae^{i2x}=2a\cos2x
\end{equation}
are denoted by $L_{t}(V)$ ($L(V)$) and $H_{t}(a)$ ($H(a))$ respectively.

\begin{remark}
The potentials (2) and (7) are PT-symmetric and even functions respectively.
The equalities in (6) show that to consider the spectrum we can use the
properties of both cases. Namely, we use the properties (5) of the
PT-symmetric potential (2) and the following properties of the even potential (7):

If $a\neq0,$ then the geometric multiplicity of of the eigenvalues of the
operators $H_{0}(a)$, $H_{\pi}(a),$ $D(a)$ and $N(a),$ called as periodic,
antiperiodic, Dirichlet and Neumann eigenvalues, is $1$ and the following
equalities hold
\begin{equation}
\sigma(D(a))\cap\sigma(N(a))=\varnothing,\text{ }\sigma(H_{0}(a))\cup
\sigma(H_{\pi}(a))=\sigma(D(a))\cup\sigma(N(a)),
\end{equation}
where $D(a)$ and $N(a)$ denote the operators generated in $L_{2}[0,\pi]$ by
(1) with potential (7) and Dirichlet and Neumann boundary conditions
respectively (see [6, 16]).

A great number of papers are devoted to the Mathieu operator $H(a).$ Here we
recall only the classical results and the results of [16] about Mathieu
operator $H(a)$ which are essentially used in this paper (see summaries 2 and
3). Moreover, in this paper we use the notations of $H(a)$. Thai is why, we
proof the statements for $H(a)$ and by (6) they continue to hold for $L(V)$ if
$a=\sqrt{1-4V^{2}}.$
\end{remark}

It is well-known that if $a$ is a real nonzero number (see [3, 6]) then all
eigenvalues of $H_{t}(a)$ for all $t\in\lbrack0,2\pi)$ are real and simple and
all gaps in the spectrum of $H(a)$ are open. These results can be stated more
precisely as follows.

\begin{summary}
Let $0<a<\infty.$ Then all eigenvalues of $H_{t}(a)$ for all $t\in
\lbrack0,2\pi)$ are real and simple and the spectrum of $H(a)$ consists of the
real intervals
\[
\Gamma_{1}=:[\lambda_{0}(a),\lambda_{1}^{-}(a)],\text{ }\Gamma_{2}%
=:[\lambda_{1}^{+}(a),\lambda_{2}^{-}(a)],\text{ }\Gamma_{3}=:[\lambda_{2}%
^{+}(a),\lambda_{3}^{-}(a)],\text{ }\Gamma_{4}=:[\lambda_{3}^{+}%
(a),\lambda_{4}^{-}(a)],...,
\]
where $\lambda_{0}(a)$, $\lambda_{2n}^{-}(a)$, $\lambda_{2n}^{+}(a)$ for
$n=1,2,...$ are the eigenvalues of $H_{0}(a)$ and $\lambda_{2n+1}^{-}(a)$,
$\lambda_{2n+1}^{+}(a)$ for $n=0,1,...$ are the eigenvalues of $H_{\pi}(a)$
and the following inequalities hold \
\[
\lambda_{0}(a)<\lambda_{1}^{-}(a)<\lambda_{1}^{+}(a)<\text{ }\lambda_{2}%
^{-}(a)<\lambda_{2}^{+}(a)<\text{ }\lambda_{3}^{-}(a)<\lambda_{3}%
^{+}(a)<,....
\]
The bands $\Gamma_{1},$ $\Gamma_{1},...$ of the spectrum $\sigma(H(a))$ are
separated by the gaps%
\[
\Delta_{1}=:(\lambda_{1}^{-}(a),\lambda_{1}^{+}(a)),\text{ }\Delta
_{2}=:(\lambda_{2}^{-}(a),\lambda_{2}^{+}(a)),\text{ }\Delta_{3}=:(\lambda
_{3}^{-}(a),\lambda_{3}^{+}(a)),....
\]
By the other notation $\Gamma_{n}=\left\{  \mu_{n}(t):t\in\lbrack
0,\pi]\right\}  ,$ where $\mu_{1}(t)<\mu_{2}(t)<...$ are the eigenvalues of
$H_{t}(a)$ called as Bloch eigenvalues corresponding to the quasimomentum $t.$
The Bloch eigenvalue $\mu_{n}(t)$ continuously depends on $t$ and $\mu
_{n}(-t)=\mu_{n}(t)$ (see (3)).

By (6) these statements continue to hold for $L_{t}(V)$ and $L(V)$
respectively if $0<V<1/2.$
\end{summary}

The case $V=1/2$ for the first time is considered in [4] and it was proved
that the spectrum is $[0,\infty).$ Thus we need to consider the spectrum of
$L(V)$ in the case $V>1/2$ which, by (6), is the considerations of
$\sigma(H(a))$ in the case $a=ic,$ $c>0.$ In [16] we obtained the following
results which are essentially used in this paper.

\begin{summary}
$(a)$ If $0<\left\vert a\right\vert \leq8/\sqrt{6}$, then all eigenvalues of
the operators $H_{\pi}(a)$ and $D(a)$ are simple. All eigenvalues of the
operator $H_{\pi}(a)$ lie on the union of $D_{2\left\vert a\right\vert
}\left(  (2n-1)^{2}\right)  $ for $n\in\mathbb{N}$, where $D_{r}%
(z)=\{\lambda\in\mathbb{C}:\left\vert \lambda-z\right\vert \leq r\}.$

$(b)$ If $0<\left\vert a\right\vert \leq4/3$ then all eigenvalues of
$H_{0}(a)$ are simple. All eigenvalues of $H_{0}(a)$ lie in the union of
$D_{\sqrt{2}\left\vert a\right\vert }(0),$ $D_{\left(  1+\sqrt{2}\right)
\left\vert a\right\vert }(4)$ and $D_{2\left\vert a\right\vert }((2n)^{2})$
for $n=2,3,....$

$(c)$ If $a\neq0,$ then the number $\lambda$\ is an eigenvalue of multiplicity
$s$ of $H_{0}(a)$ if and only if it is an eigenvalue of multiplicity $s$
either of $D(a)$ or $N(a)$. The statement continues to hold if $H_{0}(a)$ is
replaced by $H_{\pi}(a).$
\end{summary}

To easify the readability of this paper in Section 2 we discuss the main
results and give the brief and descriptive scheme of the proofs. Then in
sections 3-5 we give the rigorous mathematical proofs of the results. In order
to avoid eclipsing the essence by the technical details some calculations and
estimations are given in the Appendix.

\section{Discussion of the Main Results and Proofs}

In this section we describe the main results and give a brief scheme of some
proofs. Moreover, we describe the transfigurations of the spectra of $L(V)$
when $V$ changes from $0$ to $\infty$. If $V$ changes from $0$ to $\infty$
then $a=\sqrt{1-4V^{2}}$ moves from $1$ to $0$ over the real line and then
moves from $0$ to $\infty$ over the imaginary line. In the case $a\in
(0,\infty)$ the spectrum is described in Summary 2. \textbf{ }We consider in
detail the case $a\in I(2),$ that is, $1/2<V<\sqrt{5}/2,$ where
$I(c)=:\left\{  ix:x\in(0,c)\right\}  .$ The steps of the investigations are
the followings.

\textbf{Step 1. On the periodic and antiperiodic eigenvalues.} In Section 3 we
consider the periodic and antiperiodic eigenvalues.\textbf{ }The main results
are the followings.

$(a)$ If $a\in I(2),$ then all antiperiodic eigenvalues are nonreal and
simple. They consist of the numbers $\lambda_{1}^{+}(a),$ $\lambda_{3}%
^{+}(a),$ $\lambda_{5}^{+}(a),...,$ lying in the upper half plane and their
conjugates denoted by $\lambda_{1}^{-}(a),$ $\lambda_{3}^{-}(a),$ $\lambda
_{5}^{-}(a),....$ respectively.

$(b)$ If $a\in I(4/3)$, that is, if $1/2<V<5/6$ then all periodic eigenvalues
are real and simple and hence can be numbered as in the self-adjoint case:
\begin{equation}
\lambda_{0}(a)<\lambda_{2}^{-}(a)<\lambda_{2}^{+}(a)<\lambda_{4}%
^{-}(a)<\lambda_{4}^{+}(a)<....,
\end{equation}
(see Summary 2). Furthermore, the eigenvalues $\lambda_{2}^{+}(a),$
$\lambda_{4}^{-}(a),$ $\lambda_{4}^{+}(a)$ are real simple and satisfy (9) for
all $a\in I(2).$ However for $\lambda_{0}(a)$ and $\lambda_{2}^{-}(a)$ we
prove that there exists a unique number $ir\in I(2)$ such that $\lambda
_{0}(ir)=\lambda_{2}^{-}(ir),$ that is, $\lambda_{0}(ir)$ is the double
eigenvalue of $H_{0}(ir).$ Moreover, we prove that the number $\tfrac{1}%
{2}\left(  1+r^{2}\right)  ^{1/2}$ is the second critical number $V_{2}$. We
say that the number $ir,$ as well as $V_{2},$ is the degeneration point for
the first periodic eigenvalue, since we prove that the first periodic
eigenvalue of the potential (2) is simple for all $V\in(1/2,\sqrt
{5}/2)\backslash\left\{  V_{2}\right\}  $ and is double if $V=V_{2}.$ If
$0<a/i<r,$ then both $\lambda_{0}(a)$ and $\lambda_{2}^{-}(a)$ are real and
simple and if $r<a/i<2$ \ then $\lambda_{0}(a)$ and $\lambda_{2}^{-}(a)$ are
simple and nonreal and $\lambda_{0}(a)=\overline{\lambda_{2}^{-}(a)}.$ Thus if
$0<a/i<r$ or equivalently if $1/2<V<V_{2}$ then all periodic eigenvalues are
real simple and satisfy (9)

\textbf{Step 2. On the numerations of the Bloch eigenvalues and bands.}

In the self-adjoint case the Bloch eigenvalues and bands can be numerated in
increasing order, since they are real. It helps to describe all results for
the self-adjoint Hamiltonian. Since, in the non-self-adjoint case the above
listed quantities, in general, are not real we have the problems: $(i)$ how
numerate the Bloch eigenvalues and bands, $(ii)$ how describe the real and
nonreal parts of the bands in detail.

We prove that the Bloch eigenvalues corresponding to the quasimomentum $t$ can
be numbered by $\mu_{1}(t),$ $\mu_{2}(t),...$ such that $\mu_{n}(t)$
continuously depend on $t\in\lbrack0,\pi]$ and
\begin{align}
\mu_{1}(0)  &  =\lambda_{0}(a),\text{ }\mu_{2}(0)=\lambda_{2}^{-}(a),\text{
}\mu_{3}(0)=\lambda_{2}^{+}(a),\text{ }\mu_{4}(0)=\lambda_{4}^{-}(a),...,\\
\mu_{1}(\pi)  &  =\lambda_{1}^{-}(a),\text{ }\mu_{2}(\pi)=\lambda_{1}%
^{+}(a),\text{ }\mu_{3}(\pi)=\lambda_{3}^{-}(a),\text{ }\mu_{4}(\pi
)=\lambda_{3}^{+}(a),...,
\end{align}
Thus if $0<a/i<r$ or if $0<V<V_{2}$, then $\Gamma_{n}=\left\{  \mu_{n}%
(t):t\in\lbrack0,\pi]\right\}  $ is a continuous curve with periodic real
endpoint $\mu_{n}(0)$ and antiperiodic nonreal endpoint $\mu_{n}(\pi)$. We say
that $\Gamma_{n}$ is the $n$-th band of $\sigma(L).$ Then by (10) and (11) the
first (second) band is the continuous curve joining the periodic real
eigenvalue $\lambda_{0}(a)$ ( $\lambda_{2}^{-}(a))$ and the antiperiodic
nonreal eigenvalue $\lambda_{1}^{-}(a)$ ($\lambda_{1}^{+}(a))$.

\textbf{Step 3. On the shapes of the bands and components of the spectrum}.

We prove that the first and second bands have different shapes in the
following 3 cases:

Case 1: $0<a/i<r,$ Case 2: $a/i=r$ , Case 3: $r<a/i<2$ \ or equivalently:

\ Case 1: $1/2<V<V_{2},$ Case 2: $V=V_{2}$ , Case 3: $V_{2}<V<\sqrt{5}/2$ . In
other words in \ the cases 1-3 we describe the bands before, at and after the
second critical point.

Let us describe briefly the shapes of all bands and then stress the shapes of
the first and second bands. In Section 4, we prove that the spectrum of $L(V)$
or $H(a)$ in Case 1 has the following properties (\textbf{Pr. 1-Pr. 6}):

\textbf{Pr. 1. }\textit{The real part }$\sigma(H(a))\cap R$\textit{ of the
spectrum of }$H(a)$\textit{ consist of the intervals}\textbf{ }
\begin{equation}
I_{1}(a)=\left[  \lambda_{0}(a),\lambda_{2}^{-}(a)\right]  ,\text{ }%
I_{2}(a)=\left[  \lambda_{2}^{+}(a),\lambda_{4}^{-}(a)\right]  ,...,I_{n}%
(a)=\left[  \lambda_{2n-2}^{+}(a),\lambda_{2n}^{-}(a)\right]  ,...
\end{equation}

\textbf{Pr. 2.} \textit{For each }$n=1,2,...,$ \textit{the interval }$I_{n}%
$\textit{ is the real part of }$\Omega_{n}=:\Gamma_{2n-1}\cup\Gamma_{2n}.$

\textbf{Pr. 3.} \textit{The bands }$\Gamma_{2n-1}$\textit{ and }$\Gamma_{2n}%
$\textit{ have only one common point }$\Lambda_{n}(a)$\textit{ which is
interior point of }$I_{n}.$\textit{ Moreover, }$\Lambda_{n}(a)$\textit{ is a
double eigenvalue of }$L_{t_{n}}(V)$\textit{ for some }$t_{n}\in(0,\pi
)$\textit{ and a spectral singularity of }$L(V)$ \textit{and hence }%
\begin{equation}
\Gamma_{2n-1}\cap\Gamma_{2n}=\Lambda_{n}(a)=\mu_{2n-1}(t_{n})=\mu_{2n}%
(t_{n})\in\mathbb{R}\text{.}%
\end{equation}

\textbf{Pr. 4. }\textit{The real parts of the bands}$\ \Gamma_{2n-1}$\textit{
and }$\Gamma_{2n}$\textit{ are respectively the intervals}%

\begin{equation}
\lbrack\lambda_{2n-2}^{+},\Lambda_{n}]=\left\{  \mu_{2n-1}(t):t\in
\lbrack0,t_{n}]\right\}  \text{ }\And\lbrack\Lambda_{n},\lambda_{2n}%
^{-}]=\left\{  \mu_{2n}(t):t\in\lbrack0,t_{n}]\right\}
\end{equation}

\textbf{Pr. 5. }\textit{The nonreal parts of}$\ \Gamma_{2n-1}$\textit{ and
}$\Gamma_{2n}$\textit{ are respectively the analytic curves }%
\begin{equation}
\gamma_{2n-1}(a)=:\left\{  \mu_{2n-1}(t):t\in(t_{n},\pi]\right\}  \text{ }%
\And\text{ }\gamma_{2n}(a)=:\left\{  \mu_{2n}(t):t\in(t_{n},\pi]\right\}
\end{equation}
\textit{and }$\gamma_{2n}(a)=\left\{  \overline{\lambda}:\lambda\in
\gamma_{2n-1}(a)\right\}  $\textit{.}

Thus the bands $\Gamma_{2n-1}$ and $\Gamma_{2n}$ are joined by $\Lambda_{n}$
and hence they form together the connected subset of the spectrum. The
spectrum $\sigma(L(a))$ consist of the connected sets $\Omega_{1}=:\Gamma
_{1}\cup\Gamma_{2},$ $\Omega_{2}=:\Gamma_{3}\cup\Gamma_{4},...$Moreover in
Case 1 we prove that

\textbf{Pr. 6 }\textit{The sets }$\Omega_{1},$\textit{ }$\Omega_{2}%
,...$\textit{ are connected separated subset of} $\sigma(H(a))$.

By the last property $\Omega_{1},$\textit{ }$\Omega_{2},...$ are components of
the spectrum.

In Case 1, by \textbf{Pr. 2 }the real part of the first component $\Omega
_{1}=\Gamma_{1}\cup\Gamma_{2}$ is the closed interval $I_{1}=:\left[
\lambda_{0}(a),\lambda_{2}^{-}(a)\right]  .$ We prove that if $V$ approaches
$V_{2}$ from the left ,that is, if $a$ approaches to $ir$ from below then the
eigenvalues $\lambda_{0}(a)$ and $\lambda_{2}^{-}(a)$ get close to each other
and the length of the interval $I_{1}$ approaches zero. As a result if
$V=V_{2},$ that is, if $a=ir,$ then we get the equality $\lambda
_{0}(a)=\lambda_{2}^{-}(a)$ which means that the first and second bands
$\Gamma_{1}$ and $\Gamma_{2}$ have only one real point which is their common
point $\lambda_{0}(a)=\lambda_{2}^{-}(a)=I_{1}=\operatorname{Re}\Omega_{1}.$
Thus, in Case 2 the real parts of $\Gamma_{1}$ and $\Gamma_{2}$ is a point
$\lambda_{0}(a).$ The other parts of the bands $\Gamma_{1}$ and $\Gamma_{2}$
are \ nonreal and symmetric with respect to the real line. Then we prove that
if Case 3 occurs, then the eigenvalues $\lambda_{0}(a)$ and $\lambda_{2}%
^{-}(a)$ get off the real line and hence $I_{1}$ becomes the empty set. As a
results, the first and second bands $\Gamma_{1}$ and $\Gamma_{2}$ became the
curves symmetric with respect to the real line. Moreover, in all cases the
intervals (12) are pairwise disjoint sets. Therefore they are called the real
components of $\sigma(L(V))$ if $V\in(1/2,\sqrt{5}/2).$

Thus investigations in Step 1 and Step 3 show that, the following equivalent
mathematical definitions of the second critical point are reasonable and it is
natural to call the second critical point as the degeneration point for the
first periodic eigenvalue.

\begin{definition}
A real number $V_{2}\in(1/2,\sqrt{5}/2)$ is called the second critical point
or the degeneration point for the first periodic eigenvalue if the first real
eigenvalue of $L_{0}(V_{2})$ is a double eigenvalue.
\end{definition}

\begin{definition}
A real number $V_{2}\in(1/2,\sqrt{5}/2)$ is said to be the second critical
point or the degeneration point for the first periodic eigenvalue if the first
real component of $\sigma(L(V_{2}))$ is a point.
\end{definition}

Note that in Case 2 and Case 3 the shapes of the components $\Omega_{2},$
$\Omega_{3},...$ are as in Case 1. In this way one can prove that there exists
$k$-th critical point, denoted by $V_{k},$ such that for $\frac{1}{2}%
<V<V_{k},$ $V=V_{k}$ and $V_{k}<V<V_{k}+\varepsilon$ the set $\Omega
_{k-1}=:\Gamma_{2k-3}\cup\Gamma_{2k-2}$ have the shape as $\Omega_{1}$ in Case
1, Case 2 and Case 3 respectively.

\textbf{Step 4. Finding the approximate value of }$V_{2}$. By Summary 3(c) any
periodic eigenvalue is either Dirichlet, called as periodic Dirichlet (breifly
$PD(a)$ or PD) eigenvalue or Neumann, called as periodic Neumann ($PN(a)$ or
PN) eigenvalue. Similarly antiperiodic eigenvalue is either antiperiodic
Dirichlet (AD) or antiperiodic Neumann (AN) eigenvalues. Clearly, the
eigenfunctions corresponding to PN, PD, AD and AN eigenvalues have the forms
\[
\Psi_{PN}(x)=\frac{a_{0}}{\sqrt{2}}+\sum_{k=1}^{\infty}a_{k}\cos2kx,\text{
}\Psi_{PD}(x)=\sum_{k=1}^{\infty}b_{k}\sin2kx,
\]%
\[
\Psi_{AD}(x)=\sum_{k=1}^{\infty}c_{k}\sin(2k-1)x\text{ }\And\text{ }\Psi
_{AN}(x)=\sum_{k=1}^{\infty}d_{k}\sin2kx,
\]
respectively. Substituting these functions into equation (4) with potential
(7) we obtain the following equalities for the PN, PD, AD and AN eigenvalues
respectively
\begin{equation}
\lambda a_{0}=\sqrt{2}aa_{1},\text{ }(\lambda-4)a_{1}=a\sqrt{2}a_{0}%
+aa_{2},\text{ }(\lambda-(2k)^{2})a_{k}=aa_{k-1}+aa_{k+1},
\end{equation}%
\begin{equation}
(\lambda-4)b_{1}=ab_{2},\text{ }(\lambda-(2k)^{2})b_{k}=ab_{k-1}+ab_{k+1},
\end{equation}%
\begin{equation}
(\lambda-1)c_{1}=ac_{1}+ac_{2},\text{ }(\lambda-(2k-1)^{2})c_{k}%
=ac_{k-1}+ac_{k+1},
\end{equation}%
\begin{equation}
(\lambda-1)d_{1}=-ad_{1}+ad_{2},\text{ }(\lambda-(2k-1)^{2})d_{k}%
=ad_{k-1}+ad_{k+1}%
\end{equation}
for $k=2,3,...$, where $a_{0}\neq0,$ $b_{1}\neq0,$ $c_{1}\neq0,$ $d_{1}\neq0$
(see [3] and [16]).

\ As we noted in Step 1, the second critical point is a real number $V_{2}$
such that $\lambda_{0}(a)=\lambda_{2}^{-}(a)$, where $a=\sqrt{1-4V_{2}^{2}}.$
In Section 3 we prove that $\lambda_{0}(a)$ and $\lambda_{2}^{-}(a)$\ are the
PN eigenvalues and hence satisfy (16). Therefore to find the approximate value
of $V_{2}$ we use the following definition of $V_{2}$ which is equivalent to
the definitions 1 and 2.

\begin{definition}
A real number $V_{2}\in(1/2,\sqrt{5}/2)$ is called the second critical point
if the first real $PN(a)$ eigenvalue, where $a=\sqrt{1-4V_{2}^{2}},$ is a
double eigenvalue.
\end{definition}

Iterating (16) we obtain that the first PN eigenvalue satisfies the equality
(55). To consider the value of $a$ and hence of $V$ for which the first
periodic eigenvalue becomes double eigenvalue, we look for the double root of
(55). To find approximately the double roots of (55) we use the second
approximation (60) of (55). Using Matematica 7 we calculate the roots of the
second approximation and then estimating the remainder and using the Rouche's
theorem we prove that: $(\imath)$ both $\lambda_{0}$ and $\lambda_{2}^{-}$ are
real if $V=$ $0.8884370025,$ $(ii)$ both $\lambda_{0}$ and $\lambda_{2}^{-}$
are nonreal and complex conjugate if $V=0.8884370117.$ Moreover, we prove that
there exists unique $a$ from $I(2)$ and hence unique $V$ from $1/2<V<\sqrt
{5}/2$ such that the first eigenvalue $\lambda_{0}$ is double and this value
of $V,$ that is $V_{2},$ should be between $0.8884370025$ and $0.8884370117.$

Note that instead of the second approximation using the $m$-th approximation
of (55) for $m>2$ one can get more sharper estimation of $V_{2}.$ In the same
way we can get the arbitrary approximation of the $k$-th critical point
$V_{k}$ for $k>2$ (see Remark 5).

\section{On the Periodic and Antiperiodic Eigenvalues}

To define the numerations of the eigenvalues of $H_{\pi}(a)$ and $H_{0}(a)$
first we prove the following theorem. For this we use Summary 3 and take into
account that the spectra of $H_{0}(a),H_{\pi}(a),$ $D(a)$ and $N(a)$ for $a=0$
are
\[
\{(2k)^{2}:k=0,1,...\},\{(2k+1)^{2}:k=0,1,...\},\{k^{2}:k=1,2,...\},\{k^{2}%
:k=0,1,...\}
\]
respectively. All eigenvalues of $H_{0}(0),$ except $0,$ and $H_{\pi}(0)$ are
double, while the eigenvalues of $D(0)$ and $N(0)$ are simple.

\begin{theorem}
$(a)$ The number of eigenvalues (counting multiplicity) of operator $H_{\pi
}(a)$ lying in $D_{4}((2n-1)^{2})$ is $2$ for all $a\in I(2)$ and $n=1,2,...$
. These eigenvalues are simple. They are either real numbers (first case) or
nonreal conjugate numbers (second case). One of them is AD and the other is AN eigenvalue.

$(b)$ The statements in $(a)$ continue to hold if the operator $H_{\pi}(a)$ is
replaced by $H_{0}(a)$ and the discs $D_{4}((2n-1)^{2})$ for $n=1,2,...$ are
replaced by $D_{4}((2n)^{2})$ for $n=2,3,...$

$(c)$ The operator $H_{0}(a)$ has $3$ eigenvalues (counting multiplicity) in
$D_{6}(3)$ for all $a=I(2)$. Two of them are the PN eigenvalues and the other
is the PD eigenvalue. The PN eigenvalues are simple for $a\in I(4/3)$ and the
PD eigenvalue is simple for $a\in I(2).$ Moreover, if $a\in I(4/(1+2\sqrt
{2})),$ then $D_{\sqrt{2}\left\vert a\right\vert }(0)$ contains one PN
eigenvalue and $D_{\left(  1+\sqrt{2}\right)  \left\vert a\right\vert }(4)$
contains one PN and one PD eigenvalue.
\end{theorem}

\begin{proof}
$(a)$ It readily follows from Summary 3$(a)$ that the boundary of
$D_{4}((2n-1)^{2})$ lies in the resolvent sets of the operators $H_{\pi}(a)$
for all $a\in I(2)\cup\left\{  0\right\}  .$ Therefore the projection of
$H_{\pi}(a)$ defined by contour integration over the boundary of
$D_{4}((2n-1)^{2})$ depend continuously on $a.$ It implies that the number of
eigenvalues (counting the multiplicity) of $H_{\pi}(a)$ lying in
$D_{4}((2n-1)^{2})$ are the same for all $a=I(2)\cup\left\{  0\right\}  .$
Since $H_{\pi}(0)$ has two eigenvalues (counting the multiplicity) in
$D_{4}((2n-1)^{2}),$ the operators $H_{\pi}(a)$ has also $2$ eigenvalue.
Moreover if $a\neq0,$ then by Summary 3$(a)$ these eigenvalues are simple and
hence are different numbers. Therefore using (5) and taking into account that
if $\lambda$ lies in $D_{4}((2n-1)^{2}),$ then $\overline{\lambda}$ lies also
in $D_{4}((2n-1)^{2})$ and does not lie in $D_{4}((2m-1)^{2})$ for $m\neq n$,
we obtain that the eigenvalues lying in $D_{4}((2n-1)^{2})$ are either two
different real numbers or nonreal conjugate numbers. Instead of the operator
$H_{\pi}(a)$ using the operators $N(a)$ and $D(a),$ taking into account that
$N(0)$ and $D(0)$ have one eigenvalue in $D_{4}((2n-1)^{2}),$ and repeating
the above arguments we get the proof of the last statement.

$(b)$ Instead of Summary 3$(a)$ using Summary 3 $(b),$ repeating the proof of
$(a)$ and taking into account that $H_{0}(0)$ has $2$ eigenvalues in
$D_{4}((2n)^{2})$ for $n\geq2$ we get the proof of $(b).$

$(c)$ It is clear that if $a\in I(2)\cup\left\{  0\right\}  ,$ then the discs
$D_{\sqrt{2}\left\vert a\right\vert }(0)$ and $D_{\left(  1+\sqrt{2}\right)
\left\vert a\right\vert }(4)$ are contained in $D_{6}(3).$ If $a<4/(1+2\sqrt
{2}),$ then $D_{\sqrt{2}\left\vert a\right\vert }(0)$ and $D_{1+\sqrt
{2}\left\vert a\right\vert }(4)$ are pairwise disjoint discs and first disc
contains one eigenvalue of $N(0)$ and the second disc contains one eigenvalue
of $N(0)$ and one eigenvalue of $D(0).$ Finally note that in $D_{6}(3)$ there
exist respectively $3,2$ and $1$ eigenvalues of the operators $H_{0}(0),N(0)$
and $D(0).$ Therefore arguing as in the proof of $(a)$ and using Summary
3$(c)$ we get the proof of $(c).$
\end{proof}

\begin{notation}
By Theorem 1 $(a)$ and $(b)$ if $a=I(2)$, then the operators $H_{\pi}(a)$
and\ $H_{0}(a)$ have $2$ eigenvalues in $D_{4}((2n-1)^{2})$ for $n=1,2,...$and
$D_{4}((2n)^{2})$ for $n=2,3,...$respectively. Let us denote the eigenvalues
lying in $D_{4}(n^{2})$ by $\lambda_{n}^{-}(a)$ and $\lambda_{n}^{+}(a)$.
Moreover, in the first case (see Theorem 1$(a)$ for the first and second
cases) due to the indexing of Summary 2 we put $\lambda_{n}^{-}(a)<$
$\lambda_{n}^{+}(a)$. In the second case, without loss of generality, the
indexing can be done by the rule $\operatorname{Im}\lambda_{n}^{-}(a)<0$ and
$\operatorname{Im}\lambda_{n}^{+}(a)>0.$ Then $\lambda_{n}^{+}(a)=\overline
{\lambda_{n}^{-}(a)}.$ Three eigenvalues of the operator $H_{0}(a)$ lying in
$D_{6}(3)$ are denoted by $\lambda_{0}(a),$ $\lambda_{2}^{-}(a)$ and
$\lambda_{2}^{+}(a)$. Moreover $\lambda_{0}(a)$ and $\lambda_{2}^{-}(a)$
denote the PN eigenvalues and $\lambda_{2}^{+}(a)$ denotes the PD eigenvalue.
\end{notation}

Theorem 1 and Notation 1 imply the following.

\begin{corollary}
$(a)$ If $n\neq2,$ and $a\in I(2)$, then one of the following two cases occurs%
\begin{align}
\lambda_{n}^{\pm}(a)  &  \in\mathbb{R},\text{ }\lambda_{n}^{-}(a)<\lambda
_{n}^{+}(a)\text{ }\And\\
\lambda_{n}^{\pm}(a)  &  \notin\mathbb{R},\text{ }\operatorname{Im}\lambda
_{n}^{+}(a)>0,\text{ }\lambda_{n}^{-}(a)=\overline{\lambda_{n}^{+}(a)}.
\end{align}

\end{corollary}

Corollary 1 implies that if $n\neq2$ and $a\in I(2)$, then either $\left(
\lambda_{n}^{+}(a)-\lambda_{n}^{-}(a)\right)  \in(0,\infty)$ or $\left(
\lambda_{n}^{+}(a)-\lambda_{n}^{-}(a)\right)  \in\left\{  ix:x\in
(0,\infty)\right\}  $, that is,
\begin{equation}
\left(  \lambda_{n}^{+}(a)-\lambda_{n}^{-}(a)\right)  \in(0,\infty
)\cup\left\{  ix:x\in(0,\infty)\right\}  ,\forall a\in I(2),\forall n\neq2.
\end{equation}

Using this we prove the following.

\begin{theorem}
Let $n\neq2.$ Then $\lambda_{n}^{-}(a)$ is real (nonreal) for all $a\in I(2)$
if and only if there exists $b\in I(2)$ such that $\lambda_{n}^{-}(b)$ is a
real (nonreal) number. The statement continues to hold if $\lambda_{n}^{-}(a)$
is replaced by $\lambda_{n}^{+}(a).$
\end{theorem}

\begin{proof}
First let us prove that the set
\[
G_{n}(2)=:\left\{  \lambda_{n}^{+}(a)-\lambda_{n}^{-}(a):a\in I(2)\right\}
\]
is an subinterval of either $(0,\infty)$ or $\left\{  ix:x\in(0,\infty
)\right\}  $. Suppose to the contrary that there exist $c\in I(2)$ and $d\in
I(2)$ such that
\begin{equation}
\left(  \lambda_{n}^{+}(c)-\lambda_{n}^{-}(c)\right)  \in(0,\infty)\text{
}\And\left(  \lambda_{n}^{+}(d)-\lambda_{n}^{-}(d)\right)  \in\{ix:x\in
(0,\infty)\}.
\end{equation}
By Theorem 1, $\lambda_{n}^{+}(a)$ and $\lambda_{n}^{-}(a)$ are the simple
eigenvalues for all $a\in I(2).$ Therefore $\lambda_{n}^{+}-\lambda_{n}^{-}$
is a continuous function on $I(2).$ Thus $\left\{  \lambda_{n}^{+}%
(a)-\lambda_{n}^{-}(a):a\in\lbrack b,d]\right\}  $ is a continuous curve lying
in $(0,\infty)\cup\left\{  ix:x\in(0,\infty)\right\}  $ (see (22)) and joining
the points of $(0,\infty)$ and $\left\{  ix:x\in(0,\infty)\right\}  $ (see
(23)) which is impossible. Hence. either $G_{n}(2)\subset(0,\infty)$ or
$G_{n}(2)\subset\left\{  ix:x\in(0,\infty)\right\}  .$

Now suppose that there exists $c\in I(2)$ such that $\lambda_{n}^{-}(c)$ is a
real number. Then by Corollary 1 $\lambda_{n}^{+}(c)$ is also a real number
and $\lambda_{n}^{+}(c)-\lambda_{n}^{+}(c)\in(0,\infty)$ and hence
$G_{n}(2)\subset(0,\infty).$ It means that $\left(  \lambda_{n}^{+}%
(a)-\lambda_{n}^{-}(a)\right)  \in(0,\infty)$ for all $a\in I(2),$ that is,
both $\lambda_{n}^{+}(a)$ and $\lambda_{n}^{-}(a)$ are real numbers. In the
same way we prove the other parts of the theorem
\end{proof}

By Theorem 2 if $\lambda_{n}^{-}(a)$ is real (nonreal) for small $a$ then it
is real (nonreal) for all $a\in I(2).$

\begin{remark}
A lot of papers (see for example [5]) are devoted to the small perturbation
and asymptotic formulas when $a\rightarrow0$ (especially for real $a$) for
eigenvalues of $H_{\pi}(a)$ and\ $H_{0}(a)$ which imply that if $a$ is a small
number then $\lambda_{n}^{\pm}(a)$ is real and nonreal respectively if $n$ is
even and odd integer.

However, in order to do the paper self-contained we prove these statements in
Appendix (see Estimation 1 and Estimation 2). Namely, in Estimation 1 we prove
that%
\begin{equation}
\lambda_{1}^{+}(a)=1+a+O(a^{2}),\text{ }\lambda_{1}^{-}(a)=1-a+O(a^{2}),
\end{equation}
$\lambda_{1}^{+}(a)$ and $\lambda_{1}^{-}(a)$ are AD and AN eigenvalues
respectively. Then using Theorem 2 we prove that (see Proposition 2)
$\lambda_{2n-1}^{\pm}(a)$ is a nonreal number if $a$ is small and pure
imaginary number. Thus if $a\in I(2)$ is small, then for all odd $n$ the case
(21) occurs.

In Estimation 2 we prove that
\begin{equation}
\lambda_{0}(a)=-\tfrac{1}{2}a^{2}+O(a^{3}),\text{ }\lambda_{2}^{-}%
(a)=4+\tfrac{5}{12}a^{2}+O(a^{3}),\text{ }\lambda_{2}^{+}(a)=4-\tfrac{1}%
{12}a^{2}+O(a^{3}),
\end{equation}
where $\lambda_{0}(a)$ and $\lambda_{2}^{-}(a)$ are PN and $\lambda_{2}%
^{+}(a)$ is PD eigenvalue for small $a$. It implies that if $a\in I(2)$ is
small, then $\lambda_{0}(a),$ $\lambda_{2}^{-}(a)$ and $\lambda_{2}^{+}(a)$
are the real numbers, $\lambda_{0}(a)<$ $\lambda_{2}^{-}(a)<\lambda_{2}%
^{+}(a)$ and this notation agree with the notations of Summary 2. Then using
Theorem 2 ( see Proposition 3) we prove that $\lambda_{2n}^{\pm}(a),$ where
$n>1,$ are real numbers if $a\in I(2)$ is a small number. Thus if $a$ is small
and pure imaginary number then for even $n$ the case (20) occurs.
\end{remark}

Now using Remark 2 and Theorem 2 we consider the reality and nonreality of the
periodic and antiperiodic eigenvalues for all $a\in I(2).$

\begin{theorem}
Let $a\in I(2).$ Then $\lambda_{2n}^{-}(a)$ and $\lambda_{2n}^{+}(a)$ for all
$n=2,3,..,$ are real and%
\begin{equation}
\lambda_{4}^{-}(a)<\lambda_{4}^{+}(a)<\lambda_{6}^{-}(a)<\lambda_{6}%
^{+}(a)<.....
\end{equation}
The eigenvalues $\lambda_{2n-1}^{-}(a)$ and $\lambda_{2n-1}^{+}(a)$ are
nonreal for all $n=1,2,....$ and
\begin{equation}
\operatorname{Im}\lambda_{2n-1}^{+}(a)>0,\text{ }\lambda_{2n-1}^{-}%
(a)=\overline{\lambda_{2n-1}^{+}(a)}.
\end{equation}

\end{theorem}

\begin{proof}
By Proposition 3 (see Remark 2), $\lambda_{2n}^{\pm}(a),$ where $n>1,$ are
real number if $a$ is small. Therefore it follows from Theorem 2 that
$\lambda_{2n}^{\pm}(a)$ is real for all $a\in I(2)$. Similarly, Proposition 2
and Theorem 2 imply that $\lambda_{2n-1}^{\pm}(a)$ is nonreal for all $a\in
I(2).$ The inequalities in (26) and (27) follows from Notation 1.
\end{proof}

To consider the remaining part of the periodic eigenvalue, that is, the
eigenvalues $\lambda_{2}^{+}(a),$ $\lambda_{2}^{-}(a)$ and $\lambda_{0}(a)$ we
use the following.

\begin{proposition}
Let $d$ be a positive number. If $\lambda(a)$ is a simple eigenvalue of
$H_{0}(a)$ for all $a\in I(d),$ then it is real eigenvalue for all $a\in
I(d).$
\end{proposition}

\begin{proof}
In Proposition 3 we prove that $\lambda(a)$ is a real for small $a\in I(2).$
Let $c$ be greatest number such that $\lambda(a)$ is real for $a\in I(c)$ and
$c<d.$ Then by assumption of the proposition $\lambda(ic)$ is a simple
eigenvalue. Therefore, by general perturbation theory, $\lambda$ is analytic
function in some neighborhood of $ic$ and there exist positive constants
$\varepsilon$ and $\delta$ such that the operator $H_{0}(a)$ has only one
eigenvalue in $D_{\varepsilon}(\lambda(ic))$ whenever $\left\vert
a-ic\right\vert <\delta.$ On the other hand, by the definition of $c$ for each
$k\in\mathbb{N}$ there exists $c_{k}\in(c,c+\frac{1}{k})$ such that
$\lambda(ic_{k})$ is nonreal. Then by (5) $\overline{\lambda(ic_{k})}$ is also
periodic eigenvalue and for large value of $k$ both $\lambda(ic_{k})$ and
$\overline{\lambda(ic_{k})}$ lie in $D_{\varepsilon}(\lambda(ic))$ and
$\left\vert a-ic_{k}\right\vert <\frac{1}{k}<\delta,$ which is a contradiction.
\end{proof}

\begin{theorem}
The eigenvalue $\lambda_{2}^{+}(a)$ are real and simple for all $a\in I(2)$
and
\begin{equation}
\lambda_{2}^{+}(a)<\lambda_{4}^{-}(a)<\lambda_{4}^{+}(a).
\end{equation}

\end{theorem}

\begin{proof}
Since $\lambda_{2}^{+}(a)$ is PD eigenvalue (see Remark 2) by summaries
3$\left(  a\right)  $ and 3$(c),$ if $a\in I(2),$ then the eigenvalue
$\lambda_{2}^{+}(a)$ is simple. Therefore the propositions 1 and 3 imply that
it is real for all $a\in I(2).$ The inequalities in (28) follows from Notation 1.
\end{proof}

It remains to consider the eigenvalues $\lambda_{0}(a)$ and $\lambda_{2}%
^{-}(a).$ Now we find the upper bounds for their simplicities and realities.

\begin{theorem}
$(a)$ In $I(2)$ there exists a unique number $ir,$ called as the degeneration
point for the first periodic eigenvalue, such that $\lambda_{0}(ir)$ is a
double periodic eigenvalue and
\begin{equation}
\lambda_{0}(ir)=\lambda_{2}^{-}(ir).
\end{equation}

$(b)$ If $0<a/i<r,$ then both $\lambda_{0}(a)$ and $\lambda_{2}^{-}(a)$ are
real and simple
\begin{equation}
\lambda_{0}(a)<\lambda_{2}^{-}(a)<\lambda_{2}^{+}(a)
\end{equation}

$(c)$ The eigenvalues $\lambda_{0}(ir)$ and $\lambda_{2}^{-}(a)$ are the real numbers.

$(d)$ If $r<a/i<2$ \ then $\lambda_{0}(a)$ and $\lambda_{2}^{-}(a)$ are
nonreal and $\lambda_{0}(a)=\overline{\lambda_{2}^{-}(a)}.$
\end{theorem}

\begin{proof}
$(a)$ In Theorem 10 we prove that there exists a unique number $a\in I(2)$
such that $\lambda_{0}(a)=\lambda_{2}^{-}(a).$ In other word $\lambda_{0}(ir)$
is a multiple eigenvalue, where $r=a/i,$ and (29) holds. On the other hand,
there are only three periodic eigenvalues $\lambda_{0}(a)$, $\lambda_{2}%
^{-}(a)$ and $\lambda_{2}^{+}(a)$ in $D_{6}(3)$, where $\lambda_{0}(a)$ and
$\lambda_{2}^{-}(a)$ are Neimann eigenvalues and $\lambda_{2}^{+}(a)$ is a
Dirichlet eigenvalue (see Notation 1), and by the first equality of (8)
$\lambda_{0}(a)\neq\lambda_{2}^{+}(a)$ and $\lambda_{2}^{-}(a)\neq\lambda
_{2}^{+}(a)$ for all $a\in I(2).$ Therefore $\lambda_{0}(ir)$ and $\lambda
_{2}^{-}(ir)$ are the double eigenvalue, (29) holds and both $\lambda_{2}%
^{-}(a)$ and $\lambda_{0}(a)$ are simple for all $a\in I(2)\backslash\left\{
ir\right\}  .$

$(b)$ In the proof of $(a)$ we have proved that if $0<a/i<r,$ then both
$\lambda_{0}(a)$ and $\lambda_{2}^{-}(a)$ are simple eigenvalues. It with
Proposition 1 implies that both of them are real.

Now we prove the first inequality of (30). It follows from (25) that it holds
for small $a.$ Let $c$ be greatest number such that $c<r$ and the first
inequality in (30) holds for $a\in I(c)$. Then $\lambda_{0}(ic)=\lambda
_{2}^{-}(ic),$ since $\lambda_{0}(ic)$ and $\lambda_{2}^{-}(ic)$ are real
numbers and continuously depend on $c.$ It contradicts the simplicity of
$\lambda_{0}(ic).$

Let us prove the second inequality in (29). By (25) it holds for small $a$. On
the other hand $\lambda_{2}^{-}(a)\neq\lambda_{2}^{+}(a)$ for all $a,$ since
one of them is Dirichlet and the other is Neumann eigenvalue and (8) holds.
Therefore taking into account that the eigenvalues $\lambda_{2}^{+}(ic)$ and
$\lambda_{2}^{-}(ic)$ are real numbers and continuously depend on $c$ we get
the proof.

$\left(  c\right)  $ Since $\lambda_{0}(a)$ and $\lambda_{2}^{-}(a)$ are real
for all $a\in I(r),$ letting $c$ tend to $r$ from the left and taking into
account that the eigenvalues $\lambda_{0}(ic)$ and $\lambda_{2}^{-}%
(a)$\ continuously depend on $c$ we get the proof.

$(d)$ Let $R=:R(\lambda_{0})$ be the greatest positive number such that
$\lambda_{0}(a)$ is real for $a\in I(R)$. It follows from $(b)$ and $(c)$ and
Theorem 11$(b)$ that $r\leq R$ and $R^{2}<2.16.$ If $R>r$ then repeating the
proof of the Proposition 1 we obtain that $\lambda_{0}(iR)$ is a double
eigenvalue that contradict to $(a).$ Thus $R=r.$ Using the definition of
$R(\lambda_{0})$ we see that if a number $c$ lies in the small right
neighborhood of \ $R$ then $\lambda_{0}(ic)$ is nonreal. Let $d$ be largest
number from $(R,\infty)$ such that $\lambda_{0}(ic)$ is nonreal for $R<c<d.$
Suppose that $d<2.$ Using the Summary 1 $(e)$ and taking into account that
$\lambda_{2}^{+}(ic)$ is real (see Theorem 4) we conclude that $\lambda
_{2}^{-}(ic)=\overline{\lambda_{0}(ic)},$ for all $c\in(r,d).$ Now in the last
equality letting $c$ tend to $d,$ and using the continuity of $\lambda_{0}$
and $\lambda_{2}^{-}$ we get $\lambda_{2}^{-}(id)=\lambda_{0}(id)$ which
contradicts $(a).$
\end{proof}

\section{On the Bands and Components of the Spectrum}

In previous section we considered, in detail, the periodic eigenvalues that
will be used essentially in this section. The results of the theorems 3-5 can
be summarized as follows:

\begin{summary}
Let $ir$ be the degeneration point for the first periodic eigenvalue defined
in Theorem 5. Then the followings hold:

$(a)$ If $0<a/i<r,$ then all eigenvalues of $H_{0}(a)$ are real simple and (9) holds.

$(b)$ If $a=ir,$ then all eigenvalue of $H_{0}(a)$ are real and other from
$\lambda_{0}(a)$ and $\lambda_{2}^{-}(a)$ are simple and $\lambda
_{0}(a)=\lambda_{2}^{-}(a)<\lambda_{2}^{+}(a)<\lambda_{4}^{-}(a)<\lambda
_{4}^{+}(a)<....$

$(c)$ If $r<a/i<2,$ then all eigenvalue of $H_{0}(a)$ are simple and other
from $\lambda_{0}(a)$ and $\lambda_{2}^{-}(a)$ are real, $\lambda
_{0}(a)=\overline{\lambda_{2}^{-}(a)}$ and $\lambda_{2}^{+}(a)<\lambda_{4}%
^{-}(a)<\lambda_{4}^{+}(a)<....$

$(d)$ The statements $(a)-(c)$ continue to hold if $0,$ $2,$ $ir,$ $a$ and
$H_{0}(a)$ are replaced respectively by $1/2,$ $\sqrt{5}/2,$ $V_{2}=:\tfrac
{1}{2}\left(  1+r^{2}\right)  ^{1/2}$, $V=:\tfrac{1}{2}\left(  1-a^{2}\right)
^{1/2}$ and $L_{0}(V).$
\end{summary}

To investigate the bands and components of the spectrum we need to consider
all Bloch eigenvalues for all values of quasimomentum $t\in\lbrack0,\pi].$ In
[17] and [18] we obtained the following results formulated below as \ Summary
5 and Summary 6 (see (16) of [17] and Proposition 1, Remark 1 and Theorem 1 of [18]).

\begin{summary}
$(a)$ Bloch eigenvalues $\mu_{1}(t),$ $\mu_{2}(t),...$ can be numbered so that
$\mu_{n}(t)$ continuously depend on $t\in\lbrack0,\pi].$ Therefore $\Gamma
_{n}=\left\{  \mu_{n}(t):t\in\lbrack0,\pi]\right\}  $ is a continuous curve
and is called the $n$-th band of the spectrum.

$(b)$ $\mu_{n}(t)$ is a multiple eigenvalue of $L_{t}(q)$ of multiplicity $p$
if and only if $p$ bands of the spectrum have common point $\mu_{n}(t)$. In
particular, $\mu_{n}(t)$ is a simple eigenvalue if and only if it belong only
to one band $\Gamma_{n}$.

$(c)$ $\Gamma_{n}$ is a single open curve with the end points $\mu_{n}(0)$ and
$\mu_{n}(\pi)$.
\end{summary}

\begin{summary}
Let $q$ be PT-symmetric potential. Then

$(a)$ If $\mu_{n}(t_{1})$ and $\mu_{n}(t_{2})$ are real numbers, where $0\leq
t_{1}<t_{2}\leq\pi$ then

$\gamma:=\left\{  \mu_{n}(t):t\in\lbrack t_{1},t_{2}]\right\}  $ is an
interval of the real line with end points $\mu_{n}(t_{1})$ and $\mu_{n}%
(t_{2})$. In other words, if two real numbers $c_{1}<c_{2}$ belong to the band
$\Gamma_{n}$ then $[c_{1},c_{2}]\subset\Gamma_{n}.$

$(b)$ Two bands $\Gamma_{n}$ and $\Gamma_{m}$ may have at most one common point.
\end{summary}

\begin{notation}
If $0<a/i<r$, then, by Summary 4$(a)$ and Summary5 any eigenvalue in (9) is an
end point of only one component $\Gamma_{n}$ and for any component $\Gamma
_{n}$ there exists unique eigenvalue from (9) which is the end point of
$\Gamma_{n}.$ Thus there are one to one correspondence between bands of the
spectrum and the periodic eigenvalues (9). Without loss of generality, it can
be assumed that (10) holds. In other words, for all $a\in I(ir)$ we have
\begin{align}
\lambda_{0}(a)  &  =\Gamma_{1}\cap\sigma(H_{0}(a)),\text{ }\lambda_{2}%
^{-}(a)=\Gamma_{2}\cap\sigma(H_{0}(a)),\text{ }\\
\lambda_{2}^{+}(a)  &  =\Gamma_{3}\cap\sigma(H_{0}(a)),\text{ }\lambda_{4}%
^{-}(a)=\Gamma_{4}\cap\sigma(H_{0}(a)),\text{ }\lambda_{4}^{+}(a)=\Gamma
_{5}\cap\sigma(H_{0}(a)),...
\end{align}
Thus the equalities (31) and (32) constitute one to correspondence between
periodic eigenvalues and bands.

If $a=ir$ or $r<a/i<2,$ then by Summary 4$(b)$ and Summary 5 equality (32)
constitute one to correspondence between periodic eigenvalues $\lambda_{2}%
^{+}(a),$ $\lambda_{4}^{-}(a),$ $\lambda_{4}^{+}(a),...$ and bands $\Gamma
_{3},$ $\Gamma_{4},$ $\Gamma_{5},...$ If $a=ir$ then the first and second
bands $\Gamma_{1}$ and $\Gamma_{2}$ have common end point $\lambda
_{0}(a)=\lambda_{2}^{-}(a).$ If $r<a/i<2$ then the first and second bands
$\Gamma_{1}$ and $\Gamma_{2}$ are the bands whose one endpoints are nonreal
periodic eigenvalues $\lambda_{0}(a)$ and $\lambda_{2}^{-}(a)$ respectively.
Thus in any case the equalities (31) and (32) constitute one to correspondence
between periodic eigenvalues and bands.
\end{notation}

Note that Notation 2 with Summary 4 implies the followings.

\begin{remark}
If $0<a/i\leq r,$ then by Summary 4$(a)$ and Summary 4$(b)$ all periodic
eigenvalues are real, and hence by Notation 2, all bands of the spectrum have
a real part. If $r<a/i<2$, then by by Summary 4$(c)$ all periodic eigenvalues
except $\lambda_{0}(a)$ and $\lambda_{2}^{-}(a)$ and hence all bands except
may be $\Gamma_{1}$ and $\Gamma_{2}$ have a real part.
\end{remark}

To describe the shapes of the bands, in detail, we study the Hill discriminant
$F(\lambda)$ defined in Summary 1. First of all recall that the eigenvalues of
$H_{0}(a)$ and $H_{\pi}(a)$ are respectively the roots of $F(\lambda)=2$ and
$F(\lambda)=-2.$ It is well known [3, 14] that $F$ is an entire function and
\begin{equation}
F(\lambda)\in\mathbb{R},\text{ }\forall\lambda\in\mathbb{R},\lim
_{\lambda\rightarrow-\infty}F(\lambda)=\infty.
\end{equation}
Since $\sigma(H(a))=\left\{  \lambda\in\mathbb{C}:-2\leq F(\lambda
)\leq2\right\}  $ (see Summary 1$(c)$), it is clear that the real part of the
spectrum of $H(a)$ is%
\begin{equation}
\operatorname{Re}(\sigma(H(a)))=:\sigma(H(a))\cap\mathbb{R}=\left\{
\lambda\in\mathbb{R}:-2\leq F(\lambda)\leq2\right\}  .
\end{equation}
By (33), the set $G(F)=:\left\{  (\lambda,F(\lambda)):\lambda\in
\mathbb{R}\right\}  $ is a continuous curve in $\mathbb{C}$ called as a graph
of $F.$ Therefore, $\operatorname{Re}(\sigma(H))$ is the set of $\lambda
\in\mathbb{R}$ such that the graph lies in the strip

$\left\{  \lambda\in\mathbb{C}:-2\leq\operatorname{Im}\lambda\leq2\right\}  .$
Since all antiperiodic eigenvalues are nonreal $G(F)$ never intersect the line
$y=-2.$ It with (33) implies that
\begin{equation}
F(\lambda)>-2,\text{ }\forall\lambda\in\mathbb{R}.
\end{equation}
Now using (33)-(35) we prove \textbf{Pr. 1} of Section 2.

\begin{theorem}
If $0<a/i<r$, then the real part of the spectrum of $H(a)$ consist of the
intervals (12). In cases $a=ir$ and $r<a/i<2$, the real parts of $\sigma(H)$
are $\left\{  \lambda_{0}(a)\right\}  \cup I_{2}\cup I_{3}\cup...$ and
$I_{2}\cup I_{3}\cup...$ respectively.
\end{theorem}

\begin{proof}
First we prove the theorem for $0<a/i<r$. Since all periodic eigenvalues (9)
are real and simple the intersection of $G(F)$ and the line $y=2$ are the
points
\begin{equation}
(\lambda_{0},2),\text{ }(\lambda_{2}^{-},2),\text{ }(\lambda_{2}^{+},2),\text{
}(\lambda_{4}^{-},2),\text{ }(\lambda_{4}^{+},2),...
\end{equation}
This with (35) implies that the graph $G(H)$ may get in and out of the strip

$\left\{  \lambda\in\mathbb{C}:-2\leq\operatorname{Im}\lambda\leq2\right\}  $
at the points (36). By (9) the leftmost intersection point of the graph $G(F)$
and the line $y=2$ is $(\lambda_{0}(a),2).$ Therefore it follows from (33)
that
\begin{equation}
\sigma(H(a))\cap(-\infty,\lambda_{0}(a))=\varnothing\And F(\lambda
)>2,\forall\lambda\in(-\infty,\lambda_{0}(a)).
\end{equation}
Since $\lambda_{0}(a)$ is a simple eigenvalue we have $F^{\prime}(\lambda
_{0}(a))\neq0.$ Then the equality $F(\lambda_{0}(a))=2$ with the inequality in
(37) implies that $F^{\prime}(\lambda_{0}(a))<0,$ that is, $F(\lambda)$
decreases in some neighborhood of $\lambda_{0}(a)$, and hence $F(\lambda)<2,$
on some right neighborhood of $\lambda_{0}(a).$ Thus the graph $G(F)$ get in
of the strip $\left\{  \lambda\in\mathbb{C}:-2\leq\operatorname{Im}\lambda
\leq2\right\}  $ for the first time at the point $(\lambda_{0},2).$ Using this
and taking into account that the second intersection point of the graph $G(F)$
and the line $y=2$ is $(\lambda_{2}^{-}(a),2)$ we see that $F(\lambda)\leq2$
for all $\lambda$ from the interval $\left[  \lambda_{0}(a),\lambda_{2}%
^{-}(a)\right]  .$ This with (35) implies that
\begin{equation}
-2<F(\lambda)<2
\end{equation}
for all $\lambda\in\lbrack\lambda_{0}(a),\lambda_{2}^{-}(a))$ and
$I_{1}\subset\sigma(H(a))$.

Now let us prove that the interval $(\lambda_{2}^{-}(a),\lambda_{2}^{+}(a))$
has no common point with the spectrum. Since $\lambda_{2}^{-}(a)$ is a simple
eigenvalue we have $F^{\prime}(\lambda_{2}^{-}(a))\neq0.$ On the other hand,
(38) with $F(\lambda_{2}^{-}(a))=2$ shows that $F(\lambda)$ increases, that
is, $F(\lambda)>2$ in some right neighborhood of $\lambda_{2}^{-}(a)$. Thus
the graph $G(F)$ goes out the strip $\left\{  \lambda\in\mathbb{C}%
:-2\leq\operatorname{Im}\lambda\leq2\right\}  $ for the first time at point
$(\lambda_{2}^{-}(a),2)$ and come back to the strip at $(\lambda_{2}%
^{+}(a),2),$ since the last is the third intersection point of $G(F)$ and the
line $y=2$ (see (36)) and $F^{\prime}(\lambda_{2}^{+}(a))\neq0.$ Thus we have
proved that $(\lambda_{2}^{-}(a),\lambda_{2}^{+}(a))\cap\sigma
(H(a))=\varnothing.$ Repeating these proofs we see that $[\lambda_{2}%
^{+}(a),\lambda_{4}^{-}(a)]\subset\sigma(H(a))$ and $(\lambda_{4}%
^{-}(a),\lambda_{4}^{+}(a))\cap\sigma(H(a))=\varnothing.$ Continuing this
process we get the proof of the theorem for $0<a/i<r$.

To prove the theorem for $a=ir$ we repeat the above proof and take into
account that the line $y=2$ is the tangent to $G(F)$ at the the point $\left(
\lambda_{0}(a),2\right)  $ and the graph $G(H)$ does not get in of the strip
$\left\{  \lambda\in\mathbb{C}:-2\leq\operatorname{Im}\lambda\leq2\right\}  $
at this point. To prove the theorem for $r<a/i<2$ we also repeat the above
proof and take into account that the graph $G(H)$ get in and out of the strip
$\left\{  \lambda\in\mathbb{C}:-2\leq\operatorname{Im}\lambda\leq2\right\}  $
at the points $(\lambda_{2}^{+},2),(\lambda_{4}^{-},2),(\lambda_{4}%
^{+},2),...$.
\end{proof}

Now we find the real part and nonreal parts of each band, i.e., prove
\textbf{Pr. 2-Pr. 6.}

\begin{theorem}
If $0<a/i<r$, then for each $n=1,2,...$ \textbf{Pr. 2-Pr. 6} hold.
\end{theorem}

\begin{proof}
\textbf{The proof of Pr. 2.} By Theorem 6 to prove \textbf{Pr. 2} it is enough
to show that
\begin{equation}
I_{n}\cap\Gamma_{m}=\varnothing,\forall m\neq2n-1,2n
\end{equation}
If (39) is not true then by Notation 2 there exists $c\in(\lambda_{2n-2}%
^{+}(a),\lambda_{2n}^{-}(a))$ such that $c\in\Gamma_{m}.$ Then by Summary 6(a)
the interval with end points $c$ and $\mu_{m}(0)$ is subset of $\Gamma_{m}.$
It implies that either $\lambda_{2n-2}^{+}(a)$ or $\lambda_{2n}^{-}(a)$ belong
to $\Gamma_{m}$ which contradicts to Notation 2.

\textbf{The proof of Pr. 3. }By Theorem 6 we have
\begin{equation}
F(\lambda_{2n-2}^{+}(a))=F(\lambda_{2n}^{-}(a))=2\text{ }\And\text{
}-2<F(\lambda)<2,\forall\lambda\in(\lambda_{2n-2}^{+}(a),\lambda_{2n}^{-}(a)).
\end{equation}
Since $F$ is differentiable function by the Roll's theorem there exists

$\Lambda_{n}(a)\in(\lambda_{2n-2}^{+}(a),\lambda_{2n}^{-}(a))\subset I_{n}$
such that $F^{\prime}(\Lambda_{n}(a))=0$. It with the inequality in (40)
implies that $\Lambda_{n}(a)$ is a multiple eigenvalue of $H_{t_{n}}(a)$ for
some $t_{n}\in(0,\pi).$ On the other hand, by (39), $\Lambda_{n}%
(a)\notin\Gamma_{m}$ for all $m\neq2n-1,2n.$ Therefore $\Lambda_{n}(a)$ is a
double eigenvalue and by Summary 5(b) we have $\Lambda_{n}(a)=\mu_{2n-1}%
(t_{n})=\mu_{2n}(t_{n})\subset\Gamma_{2n-1}\cap\Gamma_{2n}.$ This with Summary
6(b) implies (13). By Proposition 2 of [17] the double eigenvalue $\Lambda
_{n}(a)$ of the operator $L_{t}(q)$ for $t\in(0,\pi)$ is the spectral
singularities of $L(q),$ where $q$ is an arbitrary periodic potential.

\textbf{The proof of Pr. 4. }By Notation 2 and (13) $\lambda_{2n-2}^{+}(a)$
and $\Lambda_{n}(a)$ belong to $\Gamma_{2n-1}$ and $\lambda_{2n-2}^{+}%
(a)=\mu_{2n-1}(0)\in\mathbb{R},$ $\Lambda_{n}(a)=\mu_{2n-1}(t_{n}%
)\in\mathbb{R}.$ Therefore using Summary 6(a) we obtain $[\lambda_{2n-2}%
^{+}(a),\Lambda_{n}(a)]=\left\{  \mu_{2n-1}(t):t\in\lbrack0,t_{n}]\right\}
\subset\Gamma_{2n-1}.$ Similarly $[\Lambda_{n}(a),\lambda_{2n}^{-}%
(a)]\subset\Gamma_{2n-1}.$ Now to prove \textbf{Pr. 4} it is enough to show
that the curves $\gamma_{2n-1}(a)$ and $\gamma_{2n}(a)$ defined in (15) lie in
$\mathbb{C}\backslash\mathbb{R}.$ We prove it in the proof of \textbf{Pr. 5}.

\textbf{The proof of Pr. 5. }We need to show that $\gamma_{2n-1}(a)$ and
$\gamma_{2n}(a)$ have no real points. We prove it for $\gamma_{2n}(a).$ The
proof for $\gamma_{2n-1}(a)$ is the same. Suppose to the contrary that there
exists $c\in(t_{n},\pi)$ such that $\mu_{2n}(c)$ is real. Then the interval
joining $\mu_{2n}(t_{n})$ and $\mu_{2n}(c)$ has overlapping subintervals with
either $[\lambda_{2n-2}^{+}(a),\Lambda_{n}(a)]$ or $[\Lambda_{n}%
(a),\lambda_{2n}^{-}(a)]$ which contradicts either Summary 6(b) or Summary 5
(c). Thus $\mu_{2n}(t)$ is nonreal for all $t\in(t_{n},\pi].$ Let us prove
that it is simple eigenvalue. Suppose that $\mu_{2n}(t)$ is multiple for some
$t\in(t_{n},\pi].$ Then there exists $m\neq2n$ such that $\mu_{2n}(t)=\mu
_{m}(t).$ Since $\mu_{2n}(t_{n})=\mu_{2n-1}(t_{n})$ (see (13)) by Summary 6(b)
$m\neq2n-1.$ Thus $m\neq2n-1,2n.$ Then the spectrum contains the continuous
curve $\gamma$ joining the real numbers $\mu_{2n}(0)$ and $\mu_{m}(0)$ and
passing though nonreal $\mu_{2n}(t).$ It with Summary 1 (e) implies that the
closed curve $\gamma\cup\overline{\gamma}$ is a subset of the spectrum which
contradicts Summary 1(d).

It remains to proof that $\gamma_{2n}(a)=\left\{  \overline{\lambda}%
:\lambda\in\gamma_{2n-1}(a)\right\}  $\textit{. }Since $\mu_{2n}(t_{n})$ is a
double eigenvalue, ($\mu_{2n}(t_{n})=\mu_{2n-1}(t_{n}))$ and the functions
$\mu_{2n}$ and $\mu_{2n-1}$ are continuous, it follows from Summary 1(e) that
there exists $\varepsilon>0$ such that $\mu_{2n}(t)=\overline{\mu_{2n-1}(t)}$
for $t\in(t_{n},t_{n}+\varepsilon).$ On the other hand, $\mu_{2n}(t)$ and
$\mu_{2n-1}(t)$ are simple and hence analytically depend on $t\in(t_{n},\pi].$
Therefore using the uniqueness of analytic continuation we complete the proof.

\textbf{The proof of Pr. 6. }By (13) and Summary 4(a) $\Omega_{n}$ is a
connected set. To prove the separability suppose to the contrary that there
exists $\lambda\in\left(  \Omega_{n}\cap\Omega_{m}\right)  $ for some $m\neq
n.$ Since the real parts of $\Omega_{n}$ and $\Omega_{m}$ are disjoint
intervals (see \textbf{Pr. 2} and (9)), $\lambda$ is a nonreal number. Then
repeating the proof of simplicity of $\mu_{2n}(t)$ which was done in the proof
of \textbf{Pr. 5 }we get a contradiction with Summary 1(d)
\end{proof}

Repeating the proof of Theorem 7 we get the following results for $r\leq
a/i<2$.

\begin{theorem}
If $a=ir$, then $I_{1}=\Lambda_{1}(a)=\mu_{1}(t_{1})=\mu_{2}(t_{1}%
)\in\mathbb{R}$ and all statements of Theorem 7 continue to hold. If
$r<a/i<2$, then $I_{1}=\varnothing$ and all statements of Theorem 7 continue
to hold for $n=2,3,....$
\end{theorem}

The arguments of these chapter give as the following result for the operator
$L(q)$ with general PT-symmetric periodic potential $q.$

\begin{theorem}
Suppose that
\[
q\in L_{1}[0,\pi],\text{ }\int_{0}^{\pi}q(x)dx=0,\text{ }q(x+\pi)=q(x),\text{
}\overline{q(-x)}=q(x)\text{ (a.e.).}%
\]
If there exists $m>0$ such that the periodic and antiperiodic eigenvalues
$\mu_{n}(0)$ and $\mu_{n}(\pi)$ for $n>m$ are nonreal numbers then there
exists $R$ such that $[R,\infty)\subset\sigma(L(q))$ and the number of the
gaps in the real part $\operatorname{Re}(\sigma(L(q)))$ of $\sigma(L(q))$ is finite.
\end{theorem}

\begin{proof}
It is well-known that $\mu_{n}(0)$ and $\mu_{n}(\pi)$ are the zeros of
$F(\lambda)=2$ and $F(\lambda)=-2$ respectively (see (3)), where the Hill
discriminant $F$ is continuous on $\mathbb{R}$, $F(\lambda)\in\mathbb{R}$ for
$\lambda\in\mathbb{R}$ (see (5)), the asymptotic formula $F(\lambda
)=2\cos\sqrt{\lambda}+O(1/\sqrt{\lambda})$ as $\lambda\rightarrow\infty$ holds
(see [3]) and $\lambda\in\sigma(L(q))$ if and only if $F(\lambda)\in
\lbrack-2,2]$ (see Summary 1 $(c)$). Thus it is enough to show that there
exists a large number $R$ such that $F(\lambda)\in\lbrack-2,2]$ for all
$\lambda\geq R.$ Suppose to the contrary that for any large positive number
$R$ there exists $\lambda_{1}\geq R$ such that $F(\lambda_{1})\notin
\lbrack-2,2].$ Without loss of generality, assume that $F(\lambda_{1})>2.$ On
the other hand, it follows from the above asymptotic formula for $F(\lambda)$
that there exists $\lambda_{2}>\lambda_{1}$ such that $F(\lambda_{2})<2.$
Since $F$ is a continuous real-valued function on $\left[  \lambda_{1}%
,\lambda_{2}\right]  ,$ there exists $\lambda\in\left[  \lambda_{1}%
,\lambda_{2}\right]  $ such that $F(\lambda)=2,$ that is, $\lambda=\mu_{n}(0)$
and hence $\mu_{n}(0)\in\mathbb{R}$. It contradicts the conditions of the Theorem.
\end{proof}

\begin{remark}
There exist a lot of asymptotic formulas for $\mu_{n}(0)$ and $\mu_{n}(\pi).$
Using those formulas one can find the conditions on $q$ such that $\mu_{n}(0)$
and $\mu_{n}(\pi)$ are nonreal number which implies that the number of gaps in
$\operatorname{Re}(\sigma(L(q)))$ is finite. Suppose that we have the formulas
$\mu_{n}(0)=a_{n}+o(n^{-\alpha})$ and $\mu_{n}(\pi)=b_{n}+o(n^{-\alpha})$ . If
there exist $m>0$ and $c>0$ such that $\ \left\vert \operatorname{Im}%
a_{n}\right\vert >cn^{-\alpha}$ and $\left\vert \operatorname{Im}%
b_{n}\right\vert >cn^{-\alpha}$ for\ all $n>m,$ then $[R,\infty)\subset
\sigma(L(q))$ for some $R.$ Besides, there exist a lot of asymptotic formulas
for the distances between neighboring periodic and antiperiodic eigenvalues
and it readily follows from (5) that the distances are either $\left\vert
2\operatorname{Im}\mu_{n}(0)\right\vert $ or $\left\vert 2\operatorname{Im}%
\mu_{n}(\pi)\right\vert $ if $\mu_{n}(0)$ and $\mu_{n}(\pi)$ are nonreal and
$n$ is a large number. Using these relations and asymptotic formulas one can
construct a large class of the potentials $q$ for which the number of gaps in
$\operatorname{Re}(\sigma(L(q)))$ is finite.
\end{remark}

\section{Finding the Second Critical Point}

In this section we find the approximate value of the second critical point.
For this we find the approximate value of the degeneration point for the first
periodic eigenvalue. Recall that it is the value of $a\in I(2)$ for which
$\lambda_{0}(a)=\lambda_{2}^{-}(a)$. Since $\lambda_{0}(a)$ and $\lambda
_{2}^{-}(a)$ are the $PN(a)$ eigenvalues lying in $D_{6}(3)$ (see Notation1)
we consider (16) for $\left\vert a\right\vert <2$ and $\left\vert
\lambda\right\vert \leq9$. The third formula in (16) can be written as
\begin{equation}
a_{k}=\frac{aa_{k-1}+aa_{k+1}}{\lambda-(2k)^{2}},\text{ }\forall k>1.
\end{equation}
\ Using it for $a_{2},$ that is, for $k=2$ in the second formula of (16) we
get
\begin{equation}
\text{ }(\lambda-4)a_{1}=a\sqrt{2}a_{0}+\frac{a^{2}a_{1}}{(\lambda-16)}%
+\frac{a^{2}a_{3}}{(\lambda-16)}.
\end{equation}
Now we use (41) in (42) as follows. In the right hand side of (42), we isolate
the term with multiplicand $a_{1}$ and use (41) for $a_{3}.$ Then in the
obtained formula we replace everywhere $a_{k}$ by the right side of (41) if
$k>1.$ In other word, the rule of usage of (41) is the following. Every time
we isolate the terms with multiplicand $a_{1}$ and do not change it, while use
(41) for $a_{2},a_{3},....$ One can readily see that the second, fourth,....,
$2m$-th usages (41) in (42) give the terms, denoted by $A_{1}(a,\lambda
)a_{1},$ $A_{2}(a,\lambda)a_{1},...,$ $A_{m}(a,\lambda)a_{1},$ with
multiplicands $a_{1}.$ Thus after $2m$ times usages we get
\begin{equation}
\text{ }(\lambda-4)a_{1}=a\sqrt{2}a_{0}+\frac{a^{2}a_{1}}{(\lambda
-16)}+\left(
{\textstyle\sum\limits_{k=1}^{m}}
A_{k}(a,\lambda)\right)  a_{1}+R_{m}(a,\lambda),
\end{equation}
where $R_{m}(\lambda)$ is the sum of the terms without multiplicand $a_{1}.$
To explain (43) and write the formulas for $A_{k}(a,\lambda)$ and
$R_{m}(a,\lambda)$ we use the indices $n_{1},n_{2},...$ whose values are
either $-1$ or $1.$ The formula (41) for $\ k=3$ and $k=3+n_{1}$ can be
written as follows
\[
a_{3}=\sum_{n_{1}=-1,1}\frac{aa_{3+n_{1}}}{\lambda-36}\And a_{3+n_{1}}%
=\sum_{n_{2}=-1,1}\frac{aa_{3+n_{1}+n_{2}}}{\lambda-(6+2n_{1})^{2}}.
\]
These two formulas give
\begin{equation}
a_{3}=\sum_{n_{2}=-1,1}\left(  \sum_{n_{1}=-1,1}\frac{a^{2}a_{3+n_{1}+n_{2}}%
}{\left(  \lambda-36\right)  \left(  \lambda-(6+2n_{1})^{2}\right)  }\right)
.
\end{equation}
If $3+n_{1}+n_{2}=1,$\ that is, if $n_{1}=-1$ and $n_{2}=-1,$ then we get the
term with multiplicand $a_{1}.$ In (44) isolating the term with multiplicand
$a_{1}$ and using (41) for $a_{3+n_{1}+n_{2}}$ when $3+n_{1}+n_{2}>1,$ and
then for $a_{3+n_{1}+n_{2}+n_{3}},$ we get
\begin{equation}
a_{3}=\frac{a^{2}a_{1}}{(\lambda-36)(\lambda-16)}+
\end{equation}%
\[
\sum_{n_{1},n_{2},n_{3},n_{4}}\frac{\left(  \lambda-36\right)  ^{-1}%
a^{4}a_{3+n_{1}+n_{2}+n_{3}+n_{4}}}{\left(  \lambda-(6+2n_{1})^{2}\right)
\left(  \lambda-(6+2n_{1}+2n_{2}\right)  ^{2})\left(  \lambda-(6+2n_{1}%
+2n_{2}+2n_{3})^{2}\right)  },
\]
where the summation is taken under condition $3+n_{1}+n_{2}>1.$ Repeating
these usage of (41) $2m$ times and then using the formula obtained for $a_{3}$
in (42) we get (43) and see $A_{k}(a,\lambda)$ and $R_{m}(a,\lambda)$ has the
form%
\begin{equation}
A_{1}(a,\lambda)=\frac{a^{4}}{(\lambda-16)^{2}(\lambda-36)},
\end{equation}%
\begin{equation}
A_{k}(a,\lambda)=\sum_{n_{1},n_{2},...,n_{2k-1}}\frac{(\lambda-16)^{-1}\left(
\lambda-36\right)  ^{-1}a^{2k+2}}{%
{\textstyle\prod\limits_{s=1,2,...,2k-1}}
[\lambda-(6+2n_{1}+2n_{2}+...+2n_{s})^{2}]}%
\end{equation}
for $k=2,3,...,$ and
\begin{equation}
R_{m}=\sum_{n_{1},n_{2},...,n_{2m}}\frac{(\lambda-16)^{-1}\left(
\lambda-36\right)  ^{-1}a^{2m+2}a_{3+n_{1}+n_{2}+...+n_{2m}}}{%
{\textstyle\prod\limits_{s=1,2,...,2m-1}}
[\lambda-(6+2n_{1}+2n_{2}+...+2n_{s})^{2}]}.
\end{equation}
Note that $A_{1}(a,\lambda)$ is obtained due to the first term in the right
side of (45), that is, obtained in the second usage of (41) by taking
$3+n_{1}+n_{2}=1.$ The term $A_{2}$ is obtained in the fourth usage of (41) if
$3+n_{1}+n_{2}+n_{3}+n_{4}=1$ and $3+n_{1}+n_{2}>1.$ The last inequality is
necessary, for doing the third and fourth usages by the rule of usage. The
term $A_{k}$ is obtained in the $2k$th usage of (41) if
\begin{align}
3+n_{1}+n_{2}+...+n_{2k}  &  =1\text{ },\\
3+n_{1}+n_{2}+...+n_{2s}  &  >1
\end{align}
for $s=1,2,...,k-1$, since the last inequalities are necessary for doing the
$\left(  2k-1\right)  $th and $2k$th usages of (41) by the rule of usage. In
other word, the summation in (47) is taken under conditions (49) and (50).
Moreover, using (50) for $s=k-1$ and taking into account that $n_{1}%
+n_{2}+...+n_{2k-2}$ is an even number from (49) we obtain $n_{2k}%
=n_{2k-1}=-1,$
\begin{equation}
n_{1}+n_{2}+...+n_{2k-2}=0.
\end{equation}
These equalities imply that
\[
\lambda-(6+2n_{1}+2n_{2}+...+2n_{2k-2})^{2}=\lambda-36,\text{ }\lambda
-(6+2n_{1}+2n_{2}+...+2n_{2k-1})^{2}=\lambda-16.
\]
Using it in (47) we get
\begin{equation}
A_{k}(a,\lambda)=\sum_{n_{1},n_{2},...,n_{2k-3}}\frac{(\lambda-16)^{-2}\left(
\lambda-36\right)  ^{-2}a^{2k+2}}{%
{\textstyle\prod\limits_{s=1,2,...,2k-3}}
[\lambda-(6+2n_{1}+2n_{2}+...+2n_{s})^{2}]}%
\end{equation}
for $k\geq2,$ where the summation is taken under conditions (50) for
$s=1,2,...,k-2$ and (51). Since $n_{2k-2}$ is $\pm1$ and does not take part in
(52) the equality (51) can be written as $n_{1}+n_{2}+...+n_{2k-3}=\pm1.$ Thus
the summation in (52) is taken under conditions%
\begin{equation}
n_{1}+n_{2}+...+n_{2k-3}=\pm1,\text{ }3+n_{1}+n_{2}+...+n_{2s}>1,\text{
}\forall s=1,2,...,k-2
\end{equation}
Then for $k=2$ we have%
\begin{equation}
A_{2}(a,\lambda)=\frac{a^{6}}{(\lambda-16)^{3}(\lambda-36)^{2}}+\frac{a^{6}%
}{(\lambda-16)^{2}(\lambda-36)^{2}(\lambda-64)}.
\end{equation}

In Appendix (see (78)) we prove that $R_{m}(a,\lambda)\rightarrow0$ as
$m\rightarrow\infty.$ Therefore in (43) letting $m$ tend to infinity, using
$a_{1}=\frac{\lambda a_{0}}{\sqrt{2}a}$ (see (16)) and then dividing by
$a_{0}$ we get
\begin{equation}
\text{ }\lambda^{2}-4\lambda-2a^{2}-\frac{a^{2}\lambda}{(\lambda-16)}-%
{\textstyle\sum\limits_{k=1}^{\infty}}
\lambda A_{k}(a,\lambda)=0,
\end{equation}
where the series in (55) converges to some analytic function (see Remark 6).

\begin{theorem}
$(a)$ If $a\in I(2),$ then equation (55) has 2 roots (counting multiplicity)
inside the circle $C_{9}(0)=:\left\{  \lambda:\left\vert \lambda\right\vert
=9\right\}  .$ These roots coincide with the $PN(a)$ eigenvalues $\lambda
_{0}(a)$ and $\lambda_{2}^{-}(a)$ defined in Notation 1.

$(b)$ For each $a\in I(2)$ the number $\lambda_{0}(a)-\lambda_{2}^{-}(a)$ is
either real or pure imaginary.

$(c)$ There exists a unique number $a\in I(2),$ such that $\lambda
_{0}(a)=\lambda_{2}^{-}(a)$.
\end{theorem}

\begin{proof}
$(a)$ The equation $\lambda^{2}-4\lambda-2a^{2}=0$ has $2$ roots in the disc
$D_{9}(0)$ and
\[
\left\vert \lambda^{2}-4\lambda-2a^{2}\right\vert >37,\text{ }\forall
\lambda\in C_{9}(0).
\]
Therefore the estimation $\left\vert a^{2}\lambda/(\lambda-16)\right\vert
<36/7$ and Remark 6 imply that (55) has two roots in $D_{9}(0)$ due to the
Rouche's theorem. Since the $PN(a)$ eigenvalues $\lambda_{0}(a)$ and
$\lambda_{2}^{-}(a)$ are also the roots of (55) lying in $D_{9}(0)$ they
coincides with those roots.

$(b)$ If $\lambda$ is a root of (55) lying in $D_{9}(0)$ then $\overline
{\lambda}$ is also is a root of (55) lying in $D_{9}(0)$ due to reality of
$a^{2}.$ Therefore the proof of $(b)$ follows from $(a).$

$(c)$ Denote by $N(a,\lambda)$ the right hand side of (55). Double root of
(55) satisfies $N(a,\lambda)=0$ and $N^{\prime}(a,\lambda)=0,$ where
$N^{\prime}(a,\lambda)$ is the derivative of $N(a,\lambda)$ with respect to
$\lambda.$ From (79)-(81) we see that the series in (55) can be differentiated
term by term and%
\begin{equation}
N^{\prime}(a,\lambda)=2\lambda-4-\frac{a^{2}}{\lambda-16}+\frac{\lambda a^{2}%
}{\left(  \lambda-16\right)  ^{2}}-%
{\textstyle\sum\limits_{k=1}^{\infty}}
(A_{k}(a,\lambda)+\lambda A_{k}^{^{\prime}}(a,\lambda)),
\end{equation}%
\begin{equation}
N^{^{\prime\prime}}(a,\lambda)=2+\frac{2a^{2}}{\left(  \lambda-16\right)
^{2}}-\frac{2\lambda a^{2}}{\left(  \lambda-16\right)  ^{3}}-%
{\textstyle\sum\limits_{k=1}^{\infty}}
(2A_{k}^{\prime}(a,\lambda)+\lambda A_{k}^{\prime\prime}(a,\lambda)).
\end{equation}
If $\left\vert \lambda\right\vert <9$ and $\left\vert a\right\vert <2,$ then
from $N^{\prime}(a,\lambda)=0,$ by using (82), we\ obtain that $\left\vert
\lambda\right\vert <3$ and $N^{\prime\prime}(a,\lambda)\neq0.$ Thus, by the
implicit function theorem, $\lambda(a)$ is an analytic function satisfying%
\begin{equation}
\lambda(a)=2+\frac{a^{2}}{2\left(  \lambda(a)-16\right)  }-\frac
{\lambda(a)a^{2}}{2\left(  \lambda(a)-16\right)  ^{2}}+\frac{1}{2}%
{\textstyle\sum\limits_{k=1}^{\infty}}
(A_{k}(a,\lambda(a))+\lambda(a)A_{k}^{^{\prime}}(a,\lambda(a)))
\end{equation}
from which we obtain $\lambda(a)=2+f(a)$ and $\left\vert f(a)\right\vert
<1/2,$ where $f$ is an analytic function. Using it in (55) we get
\begin{equation}
2+a^{2}=g(a),\text{ }\left\vert g(a)\right\vert <1,
\end{equation}
where $g$ is an analytic function. Now using the Rouche's theorem for
functions $2+a^{2}$ and $2+a^{2}-g(a)$ on the circle $C_{1/2}(\sqrt{2}i)$ we
obtain that the equation (55) has a unique double root inside the circle. It
implies that there exists unique value of $a\in I(2)$ such that $\lambda
_{0}(a)$ is a double eigenvalue.
\end{proof}

\begin{remark}
Solving (59) by the numerical methods one can find an arbitrary approximation
for the second critical point $V_{2}$. \ It is the value of $V\in(1/2,\sqrt
{5}/2)$ for which (55) for $a=\sqrt{1-4V^{2}}$ has a double eigenvalue in
$D_{9}(0).$ In other words, we find the degeneration point for the first
periodic eigenvalue, that is, the value of $a$ for which $\lambda
_{0}(a)=\lambda_{2}^{-}(a).$ These investigations show that $\left(
n+1\right)  $th critical point $V_{n+1}$ is the value of $V\in(1/2,\infty)$
for which $\lambda_{2n-2}^{+}(a)=\lambda_{2n}^{-}(a)$ for $a=\sqrt{1-4V^{2}}$.
Therefore we say that, $a$ is the degeneration point for the $(2n-1)$th
periodic eigenvalue. Note also that then $I_{n}(a)$ consists of one point (see
(12)) and after $a$ the real parts of $(2n-1)$th and $2n$th bands disappear.
The eigenvalues $\lambda_{2n-2}^{+}(a)$ and $\lambda_{2n}^{-}(a)$ are either
$PN(a)$ or $PD(a)$ eigenvalues and hence satisfy either (16) or (17).
Therefore considering these equations in the corresponding regions and
iterating the formulas (16) or (17) for $k=n$ and arguing as in the proof of
(43) (each time isolate the terms with multiplicand $a_{n}$ or $b_{n}$ and do
not change they and use (16) or (17) if $k\neq n)$ we get a formula similar to
(55). The value of $V\in(1/2,\infty)$ for which the obtained equation for
$a=\sqrt{1-4V^{2}}$ has a double root is the $n$th critical point or the
degeneration point for the $(2n-1)$th periodic eigenvalue.
\end{remark}

Now we find the approximate value of $V_{2}$ as follows. For the $m$-th
approximations we use the equation obtained from (55) by replacing the
summations from $1$ to $\infty$ with the summation from $1$ to $m.$ For the
estimation of the second critical point we use the second approximation which,
by (46) and (54) has the form\ %

\begin{equation}
\text{ }Q(a^{2},\lambda)=:\lambda^{2}-4\lambda-\frac{a^{2}\lambda}%
{(\lambda-16)}-\frac{a^{4}\lambda}{(\lambda-16)^{2}(\lambda-36)}-
\end{equation}%
\[
\frac{a^{6}\lambda}{(\lambda-16)^{3}(\lambda-36)^{2}}-\frac{a^{6}\lambda
}{(\lambda-16)^{2}(\lambda-36)^{2}(\lambda-64)}-2a^{2}=0\text{.}%
\]

\begin{theorem}
$(a)$ If $a^{2}=-2.15728123$ then $\lambda_{0}(a)$ and $\lambda_{2}(a)$ are
the real eigenvalues of $H_{0}(a)$ lying respectively inside the circles
\begin{equation}
\gamma_{1}=\left\{  \left\vert \lambda-2.088438808\right\vert
=0.00023\right\}  \And\gamma_{2}=\left\{  \left\vert \lambda
-2.088959036\right\vert =0.00023\right\}  .
\end{equation}

$(b)$ If $a^{2}=-2.157281295$ then $\lambda_{0}(a)$ and $\lambda_{2}(a)$ are
the nonreal eigenvalues of $H_{0}(a)$ lying respectively inside the circles
\begin{equation}
\gamma_{3,4}=\left\{  \left\vert \lambda-2.088698925\pm0.000232839i\right\vert
=0.00023\right\}  .
\end{equation}

$(c)$ The second critical number $V_{2}$ satisfies the inequalities
\begin{equation}
0.8884370025<V_{2}<0.8884370117
\end{equation}

\end{theorem}

\begin{proof}
$(a)$ It is clear that
\[
P(a^{2},\lambda)=:(\lambda-16)^{3}(\lambda-36)^{2}(\lambda-64)Q(a^{2}%
,\lambda),
\]
where $Q(a^{2},\lambda)$ is defined in (60), is a polynomial with respect to
$\lambda$ of order $8.$ Computing by Matematika 7 we see that the roots of
$P(-2.15728123,\lambda)$ are
\begin{align*}
\lambda_{1}  &  =2.088438808,\text{ }\lambda_{2}=2.088959036,\text{ }%
\lambda_{3}=15.85581654,\text{ }\lambda_{4}=63.99999991,\\
\lambda_{5,6}  &  =15.98321016\pm0.11878598i,\text{ }\lambda_{7,8}%
=36.00018270\pm0.00333046i.
\end{align*}
Using the decomposition
\[
Q(a,\lambda)=\frac{(\lambda-\lambda_{1})(\lambda-\lambda_{2}).....(\lambda
-\lambda_{8})}{(\lambda-16)^{3}(\lambda-36)^{2}(\lambda-64)}%
\]
by direct calculations we obtain
\begin{equation}
\left\vert Q(-2.15728123,\lambda)\right\vert >5\times10^{-8},\text{ }%
\forall\lambda\in\gamma_{1}\cup\gamma_{2},
\end{equation}
On the other hand, in the Estimation 4 of the Appendix (see (94)) we prove
that
\begin{equation}
\left\vert
{\textstyle\sum\limits_{k\geq3}}
\lambda A_{k}(-2.15728123,\lambda)\right\vert <4.735\,7\times10^{-8},\text{
}\forall\lambda\in\gamma_{1}\cup\gamma_{2}%
\end{equation}
Hence by the Rouche's theorem the equation (55) for $a^{2}=-2.15728123$ has
only one root inside of each circles $\gamma_{1}$ and $\gamma_{2}$. Therefore
using Theorem 10 and taking into account that if $\lambda$ lies inside
$\gamma_{1}$ then $\overline{\lambda}$ does not lie inside $\gamma_{2}$, we
obtain that $\lambda_{0}(a)$ and $\lambda_{2}(a)$ are the real eigenvalues
lying respectively inside $\gamma_{1}$ and $\gamma_{2}$.

$(b)$ The roots of $P(a,\lambda)$ for the cases $a^{2}=-2.157281295$ are
\begin{align*}
&  2.088698925\pm0.000232839i,\text{ }15.98321016\pm0.11878599i,\\
&  15.85581654,\text{ }36.00018270\pm0.00333046i,\text{ }63.99999991.
\end{align*}

Instead of $\gamma_{1},$ $\gamma_{2}$ using $\gamma_{3}$, $\gamma_{4}$ and
repeating the proof of $\left(  a\right)  $ we get the proof of $(b).$

$(c)$ By Theorem 10 (b), $\lambda_{0}(a)-\lambda_{2}^{-}(a)$ is either real
number or pure imaginary number. Therefore it follows from $(a)$ and $(b)$
that \ when $a^{2}$ changes from $-2.15728123$ to $-2.157281295$ then
$\lambda_{0}(a)-\lambda_{2}^{-}(a)$ moving over $x$ and $y$ axes changes from
real to pure imaginary number. Since $\lambda_{0}(a)-\lambda_{2}^{-}(a)$
continuously depend on $a\in I(2)$, there exists a real number $\left(
ir\right)  ^{2}$ between $-2.157281295$ and $-2.15728123$ such that
$\lambda_{0}(ir)=\lambda_{2}^{-}(ir).$ Thus $ir$ is the degeneration point for
the first periodic eigenvalue defined in Theorem 5. Therefore \ $V_{2}%
=\tfrac{1}{2}\left(  1+r^{2}\right)  ^{1/2}$ is the second critical point and
(63) holds.
\end{proof}

\section{Appendix: Estimations and Calculations}

\textbf{ESTIMATION\ 1. }In this estimation we prove (24) by using (18) and
(19). One of the eigenvalues $\lambda_{1}^{-}(a)$ and $\lambda_{1}^{+}(a)$ is
AN and the other is AD eigenvalue lying in $D_{4}(1)$. First let us consider
the AD eigenvalue $\lambda(a)$ lying in $D_{4}(1)$ by using (18). By Summary
3, $\lambda(a)$ is a simple eigenvalue for small $a$ and $\lambda(a)=1+O(a).$
Then by the general perturbation theory,\ the corresponding eigenfunction is
close to the eigenfunction corresponding to $\lambda(0).$ It means that
$c_{1}=1+o(1)$ as $a\rightarrow0.$ Therefore formula (18) for $k=2$ implies
that $c_{2}=O(a).$ Using it in the first formula in (18) we get $(\lambda
(a)-1)c_{1}=ac_{1}+O(a^{2}).$ Now dividing both sides of the last equality by
$c_{1}$ we get the formula $\lambda(a)=1+a+O(a^{2})$. Instead (18) using (19)
and repeating the above proof we obtain that the AN eigenvalue satisfy the
formula $\lambda(a)=1-a+O(a^{2}).$ Thus (24) is proved. The obtained formulas
with Notation 1 imply that $\lambda_{1}^{-}(a)$ and $\lambda_{1}^{+}(a)$ are
AN and AD eigenvalues and $\lambda_{1}^{+}(a)=\overline{\lambda_{1}^{-}(a)}$ .

By this way we prove the following

\begin{proposition}
The eigenvalues $\lambda_{2n-1}^{-}(a)$ and $\lambda_{2n-1}^{+}(a)$ for small
values of $a\in I(2)$ are nonreal numbers and $\lambda_{2n-1}^{+}%
(a)=\overline{\lambda_{2n-1}^{-}(a)}.$
\end{proposition}

\begin{proof}
Let us first consider the AD eigenvalue $\lambda(a)$ lying in $D_{4}(\left(
2n-1\right)  ^{2})$ by using (18). Suppose to the contrary that $\lambda(a)$
is real for some small $a\in I(2).$ Then by Theorem 2 it is real for all small
$a.$ It readily follows from Summary 3 and (18) that $\lambda(a)$\ is a simple
eigenvalue and
\[
\lambda_{2n-1}^{+}(a)=\left(  2n+1\right)  ^{2}+O(a)\text{ }\And\text{ }%
c_{n}=1+O(a),\text{ }c_{m}=O(a),\forall m\neq n
\]
To prove the proposition we iterate $\left(  2n-2\right)  $-times the formula%
\begin{equation}
(\lambda-(2n-1)^{2})c_{n}=ac_{n-1}+ac_{n+1}%
\end{equation}
(see (18) for $k=n)$ as follows. Each time isolate the terms with multiplicand
$c_{n}$ and do not change they and use the formulas
\begin{align}
c_{k}  &  =\frac{ac_{k-1}+ac_{k+1}}{\lambda-(2k-1)^{2}}\And\\
c_{1}  &  =\frac{ac_{2}}{\lambda(a)-1-a}%
\end{align}
for the terms with multiplicand $c_{k}$ \ when $k\neq n.$ After $\left(
2n-2\right)  $ times usages of (67) or (68) in (66) we obtain
\begin{equation}
\lambda(a)=(2n+1)^{2}+G_{n}(\lambda(a))+S_{n}(\lambda(a))+R_{n}(\lambda(a)),
\end{equation}
where the terms $G_{n},$ $S_{n}$ and $R_{n}$ are defined as follows. The term
$R_{n}$ is the sum of terms containing the multiplicands $c_{k}$ for $k\neq n$
that is the sum of all nonisolated terms. Since $c_{k}$ for $k\neq n$ is
$O(a)$ and $R_{n}$ is obtained after $(2n-2)$ times iterations, we have
$R_{n}(a,\lambda)=O(a^{2n}).$

Now consider the isolated terms, that is, the terms with multiplicand $c_{n}$.
Without loss of generality it can be assumed that $c_{n}=1.$ It is clear that
the isolated terms are obtained in first, third, ..., $(2n-3)$th usage. Let
$S_{n}$ be the special isolated term which is obtained by using the formulas
(67) or (68) in the following order $k=n-1,n-2,....,2,1,2,...,n-1.$ Then
$S_{n}$ has the form
\[
S_{n}=\frac{a^{2n-2}}{(\lambda(a)-1-a)F(\lambda(a))},
\]
where $F(\lambda(a))$ is the products of $\lambda(a)-(2s-1)^{2}$ for $s\neq n$
and $F(\lambda(a))\sim1$. Note that $f(a)\sim g(a)$ means that $f(a)=O(g(a))$
and $g(a)=O(f(a))$ as $a\rightarrow0.$ The multiplicand $\lambda(a)-1-a$ in
$S_{n}$ is obtained due to application of (68). It is clear that only one
isolated term, called special isoleted term, contains $\lambda(a)-1-a,$ since
for the other isolated terms we do not apply (68). The sum of other isolated
terms is denoted by $G_{n}(\lambda(a)).$ Thus $G_{n}(\lambda(a))$ is the sum
of fractions whose numerators are $a^{2k}$ for $k=1,2,...,(n-1)$ denominators
are the products of $\lambda(a)-(2s-1)^{2}$ for $s\neq n$ and hence are real
number. Using the formula $(1-a)^{-1}=1+a+a^{2}+....$ and taking into account
that $a^{2n}$ and $a^{2n+1}$ are real and nonreal number respectively, we see
that the nonreal part of the special term $S_{n}$ is of order $a^{2n-1}.$
Using this in (69) and taking into account that \ $G_{n}$ is a real number and
$R_{n}(a,\lambda)=O(a^{2n})$ we get a contradiction $a^{2n-1}=O(a^{2n}).$ Now
the proof follows from Corollary 1.
\end{proof}

\textbf{ESTIMATION\ 2. }Here we prove (25).\textbf{ }By Theorem 1, for the
small value of $a\in I(2),$ the disc $D_{\sqrt{2}\left\vert a\right\vert }(0)$
contains one PN eigenvalue denoted by $\lambda_{0}(a).$ Arguing as in the
proof of (24) we see that $\lambda_{0}(a)=O(a),$ $a_{0}=1+O(a).$ Using it in
(16) and taking into account that $a_{2}=O(a)$ we obtain
\[
\lambda_{0}(a)a_{0}=\frac{2a^{2}a_{0}}{\lambda_{0}(a)-4}+O(a^{3}).
\]
Dividing by $a_{0}$ and using $\lambda_{0}(a)=O(a)$ we get $\lambda
_{0}(a)=O(a^{2})$ and
\begin{equation}
\lambda_{0}(a)=\frac{2a^{2}}{O(a^{2})-4}+O(a^{3})=-\frac{1}{2}a^{2}+O(a^{3}).
\end{equation}

Now we prove the second formula in (25), where $\lambda_{2}^{-}(a)$ and
$\lambda_{2}^{+}(a)$ are the PN and PD eigenvalues lying in $D_{1+\sqrt
{2}\left\vert a\right\vert }(4)$ respectively, by using (16) and (17) and
$a_{1}=1+O(a),$ $b_{1}=1+O(a),$ $\lambda_{1}^{\pm}(a)=4+O(a).$ Using the first
and third equalities of (16) in the second equality of (16) and taking into
account that $a_{3}=O(a)$ we get
\[
\text{ }(\lambda_{2}^{-}(a)-4)a_{1}=\frac{2a^{2}}{\lambda_{2}^{-}(a)}%
a_{1}+\frac{a^{2}a_{1}}{(\lambda_{2}^{-}(a)-16)}+O(a^{3})
\]
Dividing by $a_{1}$ and and then iterating it we obtain
\[
\text{ }\lambda_{2}^{-}(a)=4+\frac{2a^{2}}{\lambda_{2}^{-}(a)}+\frac{a^{2}%
}{(\lambda_{2}^{-}(a)-16)}+O(a^{3})=4+\frac{5}{12}a^{2}+O(a^{3}).
\]
Now we prove the third formula in (25). Using the second formula of (17) for
$k=2$ in the first formula of (17) and taking into account that $b_{3}=O(a)$
we get
\[
(\lambda-4)b_{1}=\frac{a^{2}}{\lambda-16}b_{1}+O(a^{3})
\]
Dividing by $b_{1}$ we obtain
\[
\text{ }\lambda=4+\frac{a^{2}}{\lambda-16}+O(a^{3})=4-\frac{1}{12}%
a^{2}+O(a^{3}).
\]
Thus the formulas in (25) are proved.

To consider the eigenvalues $\lambda_{2n}^{-}(a)$ and $\lambda_{2n}^{+}(a)$
for $n>2$ we use the formulas
\begin{equation}
(\lambda-4-\frac{2a^{2}}{\lambda})a_{1}=aa_{2},\text{ }(\lambda-(2k)^{2}%
)a_{k}=aa_{k-1}+aa_{k+1},
\end{equation}%
\begin{equation}
(\lambda-4)b_{1}=ab_{2},\text{ }(\lambda-(2k)^{2})b_{k},=ab_{k-1}+ab_{k+1},
\end{equation}
where the first formula in (71) is obtained from the first and second formulas
of (16).

\begin{proposition}
The eigenvalues $\lambda_{0}(a)$ and $\lambda_{2n}^{\pm}(a),$ where
$n=1,2,...,$ for all small $a\in I(2)$ are real numbers.
\end{proposition}

\begin{proof}
It follows from (25) that $\lambda_{0}(a)$, $\lambda_{2}^{-}(a)$ and
$\lambda_{2}^{+}(a)$ are real. Indeed if at least one of them is nonreal then
by (5) its conjugate is also periodic eigenvalue lying in $D_{6}(3)$ and by
(25) differ from the other eigenvalues. It is controdiction, since by Theorem
1, $D_{6}(3)$ contains $3$ periodic eigenvalues.

Now consider $\lambda_{2n}^{\pm}(a)$ for $n>1.$ One of $\lambda_{2n}^{-}(a)$
and $\lambda_{2n}^{+}(a)$ satisfies (71) and the other satisfies (72). For
simplicity of notation suppose that $\lambda_{2n}^{-}(a)$ satisfies (71). To
prove the proposition we iterate $2n$-times the formulas (71) and (72) in the
same manner as were iterated the formula (66) in the proof of Proposition 2 .
Here to iterate (71) we also each time isolate the terms with multiplicand
$a_{n}$ (isolated terms) and do not change they and use the formulas (71) for
the term with multiplicand $a_{k}$ for $k\neq n.$ Thus iterating the second
formula in (71) (for $k=2n$) $2n$ times we see that the sum of nonisolated
term is $O(a^{2n+2})$ and get the equality%
\begin{equation}
\lambda_{2n}^{-}(a)=(2n)^{2}+G_{n}(\lambda_{2n}^{-}(a))+S_{n-1}(\lambda
_{2n}^{-}(a))+S_{n}(\lambda_{2n}^{-})+O(a^{2n+2})
\end{equation}
Here $G_{n}(\lambda_{2n}^{-})$ is the sum of isolated terms whose denominators
does not contain the multiplicand $\lambda_{2n}^{-}-4-\frac{2a^{2}}%
{\lambda_{2n}^{-}}.$ $S_{n-1}(\lambda_{2n}^{-})$ and $S_{n}(\lambda_{2n}^{-})$
are the sum of isolated term (special isolated terms) whose denominator
contain the multiplicand $\lambda_{2n}^{-}-4-\frac{2a^{2}}{\lambda_{2n}^{-}}$
and are obtained in $\left(  2n-3\right)  $-th and $\left(  2n-1\right)  $-th
iterations respectively. It is clear that $S_{n-1}(\lambda_{2n}^{-})\sim
a^{2n-2}$ and $S_{n}(\lambda_{2n}^{-})\sim a^{2n}.$

Similarly iterating the formulas (72) $2n$ times and arguing as above we get%
\begin{equation}
\lambda_{2n}^{+}(a)=(2n)^{2}+\widetilde{G}_{n}(\lambda_{2n}^{+})+\widetilde
{S}_{n-1}(\lambda_{2n}^{+})+\widetilde{S}_{n}(\lambda_{2n}^{+})+O(a^{2n+2})
\end{equation}
Here $\widetilde{G}_{n}(\lambda_{2n}^{-})$ is the sum of isolated terms whose
denominators does not contain the multiplicand $\lambda_{2n}^{-}-4.$
$\widetilde{S}_{n-1}(\lambda_{2n}^{-})$ and $\widetilde{S}_{n}(\lambda
_{2n}^{-})$ are the sum of isolated terms whose denominator contain the
multiplicand $\lambda_{2n}^{-}-4$ and are obtained in $\left(  2n-3\right)
$th and $\left(  2n-1\right)  $th iterations respectively.

Now suppose to the contrary that $\lambda_{2n}^{-}(a)$ is nonreal for some
small $a.$ Then by Theorem 2 it is nonreal for all small $a$ and $\lambda
_{2n}^{+}(a)=\overline{\lambda_{2n}^{-}(a)}.$ Then using the formulas (73) and
(74) and taking into account that $\overline{G_{n}(\lambda_{2n}^{-})}%
=G_{n}(\lambda_{2n}^{+}),$ $\overline{S_{n-1}(\lambda_{2n}^{-})}%
=S_{n-1}(\lambda_{2n}^{+}),$ $\overline{S_{n}(\lambda_{2n}^{-})}=S_{n}%
(\lambda_{2n}^{+})$ we get%
\begin{equation}
G_{n}(\lambda_{2n}^{+})+S_{n-1}(\lambda_{2n}^{+})+S_{n}(\lambda_{2n}%
^{+})=\widetilde{G}_{n}(\lambda_{2n}^{+})+\widetilde{S}_{n-1}(\lambda_{2n}%
^{+})+\widetilde{S}_{n}(\lambda_{2n}^{+})+O(a^{2n+2}).
\end{equation}
Now using the definitions of $G_{n}$ and $\widetilde{G}_{n}$ and the equality
$(1-\alpha)^{-1}=1+\alpha+O(\alpha^{2})$ we obtain%
\begin{equation}
G_{n}(\lambda_{2n}^{+})=\widetilde{G}_{n}(\lambda_{2n}^{+}),\text{ }%
S_{n}(\lambda_{2n}^{+})=\widetilde{S}_{n}(\lambda_{2n}^{+})+O(a^{2n+2}).
\end{equation}
Let us consider $S_{n-1}(\lambda_{2n}^{+}).$ It is cleat that both
$S_{n-1}(\lambda_{2n}^{+})$ and $\widetilde{S}_{n-1}(\lambda_{2n}^{+})$
contain only one term and $S_{n-1}(\lambda_{2n}^{+})$ can be obtained from
$\widetilde{S}_{n-1}(\lambda_{2n}^{+})$ by replacing the multiplicand $\left(
\lambda_{2n}^{+}-4\right)  ^{-1}$ with $\left(  \lambda_{2n}^{+}%
-4-\frac{2a^{2}}{\lambda_{2n}^{-}}\right)  ^{-1}.$ Therefore we have
\begin{equation}
S_{n-1}(\lambda_{2n}^{+})=\widetilde{S}_{n-1}(\lambda_{2n}^{+})+\widetilde
{S}_{n-1}(\lambda_{2n}^{+})\left(  \frac{2a^{2}}{\lambda_{2n}^{-}\left(
\lambda_{2n}^{-}-4\right)  ^{2}}\right)  +O(a^{2n+2})
\end{equation}
Now using (77) and (76) in (75) we get a contradiction $2a^{2}\widetilde
{S}_{n-1}(\lambda_{2n}^{+})=O(a^{2n+2})$
\end{proof}

\textbf{ESTIMATION\ 3. }Here we estimate $R_{m}(a,\lambda)$ defined in (48)
and consider the convergence of the series in (55) for $\left\vert
a\right\vert <2$ and $\left\vert \lambda\right\vert \leq9.$ First let us
consider $R_{m}(a,\lambda).$ Since the indices $n_{1},n_{2},...n_{2m}$ take
only two values the number of the summands of $R_{m}(a,\lambda)$ is not more
than $4^{m}.$ On the other hand, the largest (by absolute value) summand is
not greater than $\left\vert a^{2m+2}\left(  \lambda-16\right)  ^{-m-1}%
(\lambda-36)^{-m}\right\vert $, since the eigenfunction (16) can be normalized
so that $\left\vert a_{s}\right\vert \leq1$ for all $s.$ Therefore we have \
\begin{equation}
\left\vert R_{m}(a,\lambda)\right\vert <\left\vert \frac{a^{2}}{\lambda
-16}\right\vert \left\vert \frac{4a^{2}}{\left(  \lambda-16\right)
(\lambda-36)}\right\vert ^{m}<\frac{4}{7}\left(  \frac{16}{189}\right)  ^{m}%
\end{equation}

Now we consider $A_{k}(a,\lambda)$ defined in (52) in a similar way. Since the
summation in $A_{k}(a,\lambda)$ is taken over $n_{1},n_{2},...n_{2k-3}$ the
number of the summands in $A_{k}(a,\lambda)$ is not more than $2^{2k-3}.$ It
is clear that, the largest (by absolute value) summand is not greater than

$\left\vert a^{2k+2}\left(  \lambda-16\right)  ^{-k-1}(\lambda-36)^{-k}%
\right\vert .$ Therefore we have
\begin{equation}
\left\vert A_{k}(a,\lambda)\right\vert \leq\left\vert \frac{a^{2}}{8\left(
\lambda-16\right)  }\right\vert \left\vert \frac{4a^{2}}{\left(
\lambda-16\right)  (\lambda-36)}\right\vert ^{k}<\frac{1}{14}\left(  \frac
{16}{189}\right)  ^{k},
\end{equation}%
\begin{equation}
\left\vert A_{k}^{\prime}(a,\lambda)\right\vert \leq\left\vert \frac{\left(
2k+1\right)  a^{2}}{8\left(  \lambda-16\right)  ^{2}}\right\vert \left\vert
\frac{4a^{2}}{\left(  \lambda-16\right)  (\lambda-36)}\right\vert ^{k}%
<\frac{2k+1}{98}\left(  \frac{16}{189}\right)  ^{k}.
\end{equation}%
\begin{equation}
\left\vert A_{k}^{\prime\prime}(a,\lambda)\right\vert \leq\left\vert
\frac{\left(  2k+2\right)  \left(  2k+1\right)  a^{2}}{8\left(  \lambda
-16\right)  ^{3}}\right\vert \left\vert \frac{4a^{2}}{\left(  \lambda
-16\right)  (\lambda-36)}\right\vert ^{k}<\frac{4k^{2}+6k+2}{686}\left(
\frac{16}{189}\right)  ^{k}.
\end{equation}

\begin{remark}
From (79) and (80) immediately follows that if \ $\left\vert a\right\vert
\leq2$ and $\left\vert \lambda\right\vert \leq9$, then the series in (55)
converges uniformly to some analytic function on the disc $\left\{  \lambda
\in\mathbb{C}:\left\vert \lambda\right\vert \leq9\right\}  .$ Moreover using
(79)-(81) by direct calculations we get%
\begin{equation}%
{\textstyle\sum\limits_{k\geq1}}
\left\vert A_{k}(a,\lambda)\right\vert <\frac{1}{100},\text{ }%
{\textstyle\sum\limits_{k\geq1}}
\left\vert A_{k}^{\prime}(a,\lambda)\right\vert <\frac{1}{200},\text{ }%
{\textstyle\sum\limits_{k\geq1}}
\left\vert A_{k}^{\prime\prime}(a,\lambda)\right\vert <\frac{1}{300},
\end{equation}

\end{remark}

\textbf{ESTIMATION\ 4. }Here we estimate $A_{k}(a,\lambda),$ in detail, when
\begin{equation}
\lambda\in\left(  \gamma_{1}\cup\gamma_{2}\cup\gamma_{3}\cup\gamma_{4}\right)
\And\text{ }2.156<-a^{2}<2.158.\text{ }%
\end{equation}
It is clear that if (83) holds then
\begin{equation}
\frac{2.15}{\left\vert 2-\left(  2s\right)  ^{2}\right\vert }<\frac{-a^{2}%
}{\left\vert \lambda-\left(  2s\right)  ^{2}\right\vert }<\frac{2.16}%
{\left\vert 2.1-\left(  2s\right)  ^{2}\right\vert }%
\end{equation}
for $s=2,3,....$ . It with (79) implies that
\[
\text{ }\left\vert A_{k}(a,\lambda)\right\vert <\left\vert \frac{\left(
2.16\right)  }{8(2.1-16)}\right\vert \left\vert \frac{4\left(  2.16\right)
}{(2.1-16)(2.1-36)}\right\vert ^{k}.
\]
Therefore, using the geometric series formula by direct calculations in SWP we
obtain%
\begin{equation}%
{\textstyle\sum\limits_{k>4}}
\left\vert A_{k}(a,\lambda)\right\vert <4.101\times10^{-11}.
\end{equation}

Now we estimate $A_{3}(a,\lambda)$ and $A_{4}(a,\lambda).$ To easify the
application of (84) we redonete $A_{k}(a,\lambda)$ by $A_{k}(a^{2},\lambda).$
It follows from (52) and (53) that
\begin{equation}
A_{3}(a^{2},\lambda)=\sum_{i,j,s}\frac{(\lambda-16)^{-2}\left(  \lambda
-36\right)  ^{-2}a^{8}}{(\lambda-(6+2i)^{2})(\lambda-(6+2i+2j)^{2}%
)(\lambda-(6+2i+2j+2s)^{2})},
\end{equation}
where the summation is taken under conditions $3+i+j>1$ and $i+j+s=\pm1.$
Therefore $A_{3}(a^{2},\lambda)$ consist of $5$ summand and $2$ of them are
the same. Namely,%
\begin{equation}
\text{ }A_{3}(a^{2},\lambda)=\frac{a^{8}}{(\lambda-16)^{4}(\lambda-36)^{3}%
}+\frac{a^{8}}{(\lambda-16)^{2}(\lambda-36)^{3}(\lambda-64)^{2}}+\nonumber
\end{equation}%
\[
\frac{2a^{8}}{(\lambda-16)^{3}(\lambda-36)^{3}(\lambda-64)}+\frac{a^{8}%
}{(\lambda-16)^{2}(\lambda-36)^{2}(\lambda-64)^{2}(\lambda-100)}.
\]
Using (84) we see that
\begin{equation}
\text{ }\left\vert A_{3}(a^{2},\lambda)\right\vert <-A_{3}%
(2.16,2.1)<2.2707\times10^{-8}%
\end{equation}

It remains to estimate $A_{4}(a^{2},\lambda)$. Let $C_{4}(a^{2},\lambda)$ and
$D_{4}(a^{2},\lambda)$ be respectively the sum of the terms of $A_{4}%
(a^{2},\lambda)$ subject to the constraints $n_{1}=-1,n_{2}=1$ and
$n_{1}=1,n_{2}=-1.$ In (52) replacing $n_{3},n_{4}$ and $n_{5}$ respectively
by $i,j$ and $s$ one can readily see that
\begin{equation}
C_{4}(a^{2},\lambda)=\frac{a^{2}}{(\lambda-16)(\lambda-36)}A_{3}(a^{2}%
,\lambda),\text{ }D_{4}(a^{2},\lambda)=\frac{a^{2}}{(\lambda-64)(\lambda
-36)}A_{3}(a^{2},\lambda).
\end{equation}
\ Now let us consider the remaining terms of $A_{4}(a^{2},\lambda),$ that is,
the terms of

$A_{4}(a^{2},\lambda)-C_{4}(a^{2},\lambda)-D_{4}(a^{2},\lambda)=:E_{4}%
(a^{2},\lambda).$ Using the definition of $C_{4}(a^{2},\lambda)$ and
$D_{4}(a^{2},\lambda)$ and taking into account (53) we see that the terms of
$E_{4}(a^{2},\lambda)$ are obtained subject to the constraint $n_{1}=1$,
$n_{2}=1$ and$\ n_{3}+n_{4}+n_{5}$ is either $-1$ or $-3$. Therefore the
triple $(n_{3},n_{4},n_{5})$ is either $(1,-1,-1)$ or $(-1,1,-1)$ or
$(-1,-1,1)$ or $(-1,-1,-1).$ Thus
\begin{equation}
E_{4}(a^{2},\lambda)=\frac{a^{10}}{(\lambda-16)^{2}\left(  \lambda-36\right)
^{2}\left(  \lambda-64\right)  ^{2}\left(  \lambda-100\right)  ^{2}\left(
\lambda-144\right)  }+
\end{equation}%
\[
\frac{a^{10}}{(\lambda-16)^{2}\left(  \lambda-36\right)  ^{2}\left(
\lambda-64\right)  ^{3}\left(  \lambda-100\right)  ^{2}}+\frac{a^{10}%
}{(\lambda-16)^{2}\left(  \lambda-36\right)  ^{3}\left(  \lambda-64\right)
^{3}\left(  \lambda-100\right)  }+
\]%
\[
\frac{a^{10}}{(\lambda-16)^{3}\left(  \lambda-36\right)  ^{3}\left(
\lambda-64\right)  ^{2}\left(  \lambda-100\right)  }.
\]
and by (88) we have
\begin{equation}
A_{3}(a^{2},\lambda)+A_{4}(a^{2},\lambda)=E_{4}(a^{2},\lambda)+\left(
1+F(a^{2},\lambda)\right)  A_{3}(a^{2},\lambda),
\end{equation}
where $F(a^{2},\lambda)=\dfrac{a^{2}}{(\lambda-16)(\lambda-36)}+\dfrac{a^{2}%
}{(\lambda-16)(\lambda-36)}$

Since the summands in $E_{4}(-2,16,2.1)$ are positive number we have%
\begin{equation}
\left\vert E_{4}(a^{2},\lambda)\right\vert <E_{4}(-2,16,2.1).
\end{equation}
Using the first inequality of (84) and taking into account that the summands
in $F(-2.15,2)$ are negative numbers and $\left\vert F(a^{2},\lambda
)\right\vert <1$ we obtain
\begin{equation}
\left\vert 1+F(a^{2},\lambda)\right\vert <1+F(-2.15,2).
\end{equation}
Thus using (90)-(92) we conclude that
\begin{equation}
\left\vert A_{3}(a^{2},\lambda)+A_{4}(a^{2},\lambda)\right\vert <E_{4}%
(-2,16,2.1)+\left(  1+F(-2,15,2)\right)  \left\vert A_{3}%
(-2,16,2.1)\right\vert .
\end{equation}
Calculating the right-hand side of (93) by SWP we get%
\[
\left\vert A_{3}(a^{2},\lambda)+A_{4}(a^{2},\lambda)\right\vert
<1.600\,9\times10^{-12}+(0.991)2.2707\times10^{-8}<2.251\times10^{-8}%
\]
This with (85) implies that
\begin{equation}
\left\vert
{\textstyle\sum\limits_{k\geq3}}
\lambda A_{k}(a,\lambda)\right\vert <2.1\left(  4.101\times10^{-11}%
+2.251\times10^{-8}\right)  <4.\,\allowbreak735\,7\times10^{-8}.
\end{equation}

\end{document}